\documentclass[12pt,english]{article}
\usepackage[T1]{fontenc}
\usepackage[utf8]{inputenc}
\usepackage{geometry}
\usepackage[skip=\medskipamount]{parskip}
\usepackage{babel}
\usepackage{array}
\usepackage{mathtools}
\usepackage{multirow}
\usepackage{varwidth}
\usepackage{amsmath}
\usepackage{amsthm}
\usepackage{amssymb}
\usepackage{stmaryrd}
\usepackage[bookmarks=true,bookmarksnumbered=false,bookmarksopen=false,
 breaklinks=false,pdfborder={0 0 0},pdfborderstyle={},backref=false,colorlinks=false]
 {hyperref}
\hypersetup{pdftitle={Computing higher limits over the fusion orbit category via amalgams},
 pdfauthor={Marco Praderio Bova}}

\makeatletter

\newcommand{\mathcircumflex}[0]{\mbox{\^{}}}

\providecommand{\tabularnewline}{\\}
\newenvironment{cellvarwidth}[1][t]
    {\begin{varwidth}[#1]{\linewidth}}
    {\@finalstrut\@arstrutbox\end{varwidth}}

\theoremstyle{plain}
\newtheorem*{conjecture*}{\protect\conjecturename}
\newcommand\thmsname{\protect\theoremname}
\newcommand\nm@thmtype{theorem}
\theoremstyle{plain}

\newenvironment{namedthm}[1][Undefined Theorem Name]{
  \ifx{#1}{Undefined Theorem Name}\renewcommand\nm@thmtype{theorem*}
  \else\renewcommand\thmsname{#1}\renewcommand\nm@thmtype{namedtheorem}
  \fi
  \begin{\nm@thmtype}}
  {\end{\nm@thmtype}}
\theoremstyle{remark}
\newtheorem*{acknowledgement*}{\protect\acknowledgementname}
\theoremstyle{plain}
\newtheorem{thm}{\protect\theoremname}[section]
\theoremstyle{definition}
\newtheorem{defn}[thm]{\protect\definitionname}
\theoremstyle{remark}
\newtheorem{notation}[thm]{\protect\notationname}
\theoremstyle{definition}
\newtheorem{example}[thm]{\protect\examplename}
\theoremstyle{plain}
\newtheorem{prop}[thm]{\protect\propositionname}
\theoremstyle{plain}
\newtheorem{lem}[thm]{\protect\lemmaname}
\theoremstyle{plain}
\newtheorem{cor}[thm]{\protect\corollaryname}
\theoremstyle{remark}
\newtheorem{rem}[thm]{\protect\remarkname}

\newcommand{\Orbitize}[1]{\mathcal{O}\left(#1\right)}

\date{}

\newcommand{\Hom}{\operatorname{Hom}}
\newcommand{\Nat}{\operatorname{Nat}}
\newcommand{\limn}[1][n]{\operatorname{lim}^{#1}}
\newcommand{\colim}{\operatorname{colim}}
\newcommand{\Rep}{\operatorname{Rep}}
\newcommand{\Ext}{\operatorname{Ext}}

\newcommand{\Tor}{\operatorname{Tor}}
\newcommand{\Tot}{\operatorname{Tot}}
\newcommand{\Aut}{\operatorname{Aut}}
\newcommand{\Inn}{\operatorname{Inn}}

\newcommand{\Id}{\operatorname{Id}}

\newcommand{\Ob}{\operatorname{Ob}}
\newcommand{\Syl}{\operatorname{Syl}}
\newcommand{\Sym}{\operatorname{Sym}}
\newcommand{\Sol}{\operatorname{Sol}}
\newcommand{\Spin}{\operatorname{Spin}}

\renewcommand{\mod}{\operatorname{-mod}}
\newcommand{\hocolim}{\operatorname{hocolim}}
\newcommand{\Ab}{\operatorname{Ab}}

\newcommand{\C}{\mathcal{C}}

\newcommand{\lui}[2]{\prescript{#1}{}{#2}}

\newcommand{\R}{\mathcal{R}}
\newcommand{\Fp}[1][p]{\mathbb{F}_{#1}}
\newcommand{\Fd}{\Fp[2]}
\newcommand{\Ff}{\Fp[5]}
\newcommand{\Fpb}[1][p]{\bold{F}_{#1}}
\newcommand{\Ffb}{\Fpb[5]}

\newcommand{\F}{\mathcal{F}}
\renewcommand{\L}{\mathcal{L}}
\newcommand{\OF}[1][\F]{\Orbitize{#1}}
\newcommand{\OFc}[1][\F]{\Orbitize{#1^c}}
\newcommand{\OFD}[2]{\mathcal{O}_{#2}\left(#1\right)}
\newcommand{\OFC}[1][\F]{\OFD{#1}{\C}}
\newcommand{\Fc}{\F^{\text{c}}}
\newcommand{\Fcr}{\F^{\text{cr}}}
\newcommand{\FAB}[2]{\F_{#1}\left(#2\right)}
\newcommand{\FSG}[1][G]{\FAB{S}{#1}}

\newcommand{\T}[1][T]{\mathcal{#1}}
\newcommand{\G}[1][G]{\mathcal{#1}}
\newcommand{\FG}{\mathcal{FG}}
\renewcommand{\Rep}{\operatorname{Rep}}

\newcommand{\Grph}{\operatorname{Grph}}
\newcommand{\Grp}{\operatorname{Grp}}

\newcommand{\CGp}[1][p]{C_G^{#1',#1}}

\makeatother

\usepackage[style=numeric,doi=false,isbn=false,url=false,eprint=false, date=year]{biblatex}
\providecommand{\acknowledgementname}{Acknowledgement}
\providecommand{\conjecturename}{Conjecture}
\providecommand{\corollaryname}{Corollary}
\providecommand{\definitionname}{Definition}
\providecommand{\examplename}{Example}
\providecommand{\lemmaname}{Lemma}
\providecommand{\notationname}{Notation}
\providecommand{\propositionname}{Proposition}
\providecommand{\remarkname}{Remark}
\providecommand{\theoremname}{Theorem}

\begin{document}
\title{Computing higher limits over the fusion orbit category via amalgams}
\author{Marco Praderio Bova}
\maketitle
\begin{abstract}
We study higher limits over the centric orbit category of a fusion
system realized by an amalgamated product. In so doing we provide
a novel technique for studying the Diaz-Park sharpness conjecture
(see \cite{DiazPark15}) and prove this conjecture (in the case of
the cohomology Mackey functors) for all the Clelland-Parker and Parker-Stroth
fusion systems. This complements the work started by Henke, Libmand
and Lynd in \cite{HLL23}. We further use the developed technique
to study the Benson-Solomon fusion systems thus relating higher limits
over the centric fusion orbit category of these systems with the signalizer
functors described in \cite{AschCher10}. We believe that the proposed
technique can, in future work, be used as a first step in an induction
argument that can bring us closer to providing an answer to this conjecture.
\end{abstract}

\section{Introduction}

Saturated fusion systems are a type of category that was first introduced
by Puig (see \cite{Puig06}) as a common framework to study fusion
of both $p$-subgroups and blocks of a finite group. Although originally
intended as a tool for modular representation theory, fusion systems
have found multiple applications in various other areas of both algebra
and topology (see \cite{AKO11} for a more detail). It was first conjectured
by Broto, Levi and Oliver (see \cite{BLO03}) and then proved by Chermack
(see \cite{Cher13}) that every saturated fusion system $\F$ has
a unique associated ``\emph{$p$-completed classifying space}''
$B\F$. This leads to the existence of a Bousfield-Kahn spectral sequence
of the form
\begin{equation}
\limn[i]_{\OFc}\left(H^{j}\left(-;\Fp\right)\right)\Rightarrow H^{i+j}\left(B\F;\Fp\right),\label{eq:cohomology-spectral-sequence}
\end{equation}
where $\Fp$ denotes the finite field of $p$-elements and $\OFc$
is the centric orbit category of $\F$ (see Definition \ref{def:general-fusion-system-notions}(\ref{enu:def-centric-orbit-category})).
Similarly, for every finite group $G$, we have the following spectral
sequence 
\begin{equation}
\limn[i]_{\mathcal{O}_{p}^{c}\left(G\right)}\left(H^{j}\left(-;\Fp\right)\right)\Rightarrow H^{i+j}\left(BG;\Fp\right).\label{eq:cohomology-spectral-sequence-groups}
\end{equation}
Dwyer proved in \cite[Theorem 10.3]{Dwyer98} that the spectral sequence
of Equation (\ref{eq:cohomology-spectral-sequence-groups}) is sharp.
That is $\limn[i]_{\mathcal{O}_{p}^{c}\left(G\right)}\left(H^{j}\left(-;\Fp\right)\right)=0$
for every $i\ge1$. As an immediate consequence, we obtain the isomorphism
$\lim_{\mathcal{O}_{p}^{c}\left(G\right)}\left(H^{n}\left(-;\Fp\right)\right)\cong H^{n}\left(BG;\Fp\right)$).
We know from \cite[Theorem 5.8]{BLO03}, that a similar stable elements
theorem also holds for saturated fusion systems. It is therefore natural
to ask weather the spectral sequence of Equation (\ref{eq:cohomology-spectral-sequence})
is also sharp. This is one of the questions listed in \cite[Section III.7]{AKO11}
and remains open at the time of writing. Based on a previous result
of Jackowsky and McClure (see \cite[Corollary 5.16]{JackowskiMcClure92}),
Diaz and Park conjectured in \cite{DiazPark15} that a stronger result
actually holds.
\begin{conjecture*}[Sharpness for fusion systems]
Let $S$ be a finite $p$-group, let $\F$ be a saturated fusion
system over $S$ and let $M=\left(M_{*},M^{*}\right)$ be a Mackey
functor over $\F$ with coefficients in $\Fp$ (see \cite[Definition 2.1]{DiazPark15}).
Then $\limn_{\OFc}\left(M^{*}\downarrow_{\OFc}^{\OF}\right)=0$ for
every $n\ge1$.
\end{conjecture*}
Since any cohomology functor can be seen as the contravariant part
of a Mackey functor, then the above conjecture is in fact stronger
than sharpness of the spectral sequence in Equation (\ref{eq:cohomology-spectral-sequence}).
Diaz and Park proved in \cite[Theorem B]{DiazPark15} that the sharpness
conjecture holds when $\F$ is realizable (see Definition \ref{def:general-fusion-system-notions}(\ref{enu:realizable-fusion-system})).
Therefore, a negative result to the conjecture would provide a method
of distinguishing between realizable and exotic fusion systems without
relying on the classification of finite simple groups. This would
solve another of the open problems listed in \cite[Section III.7]{AKO11}.
Therefore, regardless of the result, an answer to this conjecture
would be interesting of interest. It is therefore not surprising that
a large amount of articles have been written studying it (see for
instance \cites{CarrDiaz25}{DiazPark15}{GlaLyn25}{GraMar23}{HLL23}{Yal22}).
Interestingly, at the time of writing and to the best of the author's
knowledge, all known results solving the conjecture for particular
families of fusion systems, can be seen as (indirect) applications
of \cite[Corollary 5.16]{JackowskiMcClure92}. The results presented
in this article are therefore novel not only in the sense that they
prove sharpness in cohomology for all Clelland-Parker fusion systems
(complementing \cite[Theorems 1.1 and 5.2]{HLL23}), but also in the
sense that we do so without relying on \cite[Corollary 5.16]{JackowskiMcClure92}
(except for mentions of sharpness for realizable fusion systems).

\medskip{}

Let $B,G_{1},G_{2}$ be finite groups with $B\le G_{1},G_{2}$, let
$G:=G_{1}*_{B}G_{2}$ be the resulting amalgamated product and let
$\tilde{\T}$ be the tree associated to this amalgam (see Definition
\ref{def:orbit-graph}). Since $\tilde{\T}$ admits a natural $G$-action,
it is possible to define for every $P\le G$ both the subgraph $\tilde{\T}^{P}$
of points fixed under the action of $P$ and the quotient graph $\tilde{\T}^{P}/C_{G}\left(P\right)$.
By viewing graphs as $CW$-complexes in the natural way we can then
define (see Definition \ref{def:CGp}) 
\[
\CGp\left(P\right):=H_{1}\left(\tilde{\T}^{P}/C_{G}\left(P\right);\Fp\right).
\]
Take $S'\in\Syl_{p}\left(B\right)$ and $S\ge S'$ such that $S\in\Syl_{p}\left(G_{1}\right)$.
Assume that $p\not|\left[G_{2}:B\right]$ and, in particular, $S'\in\Syl_{p}\left(G_{2}\right)$.
We can define the fusion systems
\begin{align*}
\F_{\boldsymbol{1}} & :=\FSG[G_{1}], & \F_{\boldsymbol{2}} & :=\FAB{S'}{G_{2}},\\
\F_{\boldsymbol{e}} & :=\FAB{S'}{B}, & \F & :=\FSG,
\end{align*}
where we are viewing $G_{1}$ and $G_{2}$ (and therefore also $B$,
$S$ and $S'$) as subgroups of $G$. Let $\mathcal{C}$ be the family
of $\F$-centric subgroups of $S$. For every functor $M:\OFC^{\text{op}}\to\Fp\mod{}$
and every $x\in\left\{ \boldsymbol{1},\boldsymbol{2},\boldsymbol{e}\right\} $
define the $\Fp$-module
\[
M^{\F_{x}}:=\lim_{\OFC[\F_{x}]}\left(M\downarrow_{\OFC[\F_{x}]}^{\OFC}\right),
\]
which we identify with its natural projection into $M\left(S'\right)$
(or into $M\left(S\right)$ if $x=\boldsymbol{1}$). Denoting by $\overline{\iota_{S'}^{S}}$
the universal map from $M^{\F_{\boldsymbol{1}}}$ to $M^{\F_{\boldsymbol{e}}}$,
we have that both $\overline{\iota_{S'}^{S}}\left(M^{\F_{\boldsymbol{1}}}\right)$
and $M^{\F_{\boldsymbol{1}}}$ are subgroups of $M^{\F_{\boldsymbol{e}}}$.
Using spectral sequences and sharpness for realizable fusion systems,
we prove in Section \ref{sec:An-exact-sequence-for=000020higher-limits}
that the following holds.

\phantomsection
\label{thm:A}
\begin{namedthm}[Theorem A]
Let $M^{*}$ be the contravariant part of a Mackey functor over $\F$
with coefficients in $\Fp$ and write $M:=M^{*}\downarrow_{\OFc}^{\OF}$.
Assume that $\Fcr_{x}\subseteq\mathcal{C}$ (see Definition \ref{def:general-fusion-system-notions}(\ref{enu:F-centric-radical}))
for every $x\in\left\{ \boldsymbol{1},\boldsymbol{2},\boldsymbol{e}\right\} $.
Then:
\begin{enumerate}
\item For every $n\ge1$ there exists an isomorphism
\[
\Ext_{\Fp\OFc}^{n}\left(\CGp,M\right)\cong\limn[n+2]_{\OFc}\left(M\right),
\]
where, o the left hand side, we use the equivalence of categories
$\Fp\OFc\mod{}\cong\Fp\mod{}^{\OFc^{\text{op}}}$ to view contravariant
functors as modules.
\item There exists an exact sequence of the form
\[
0\to\underset{\OFc}{\limn[1]}\left(M\right)\to M^{\F_{\boldsymbol{e}}}/\left(\overline{\iota_{S'}^{S}}\left(M^{\F_{\boldsymbol{1}}}\right)+M^{\F_{\boldsymbol{2}}}\right)\to\underset{\OFc}{\Nat}\left(\CGp,M\right)\to\underset{\OFc}{\limn[2]}\left(M\right)\to0.
\]
\end{enumerate}
\end{namedthm}
There exist numerous results (see \cites{BLO06}{ClellandParker10}{Semeraro14}),
that prove that a fusion system such as $\F$ is saturated when the
graph $\tilde{\T}^{P}/C_{G}\left(P\right)$ (or similarly constructed
graphs) is a tree for every $P\in\mathcal{C}$ (see \cite[Theorem 3.2]{ClellandParker10}
for more detail). When this happens, we also have that $\CGp\left(P\right)=0$
and we deduce from Theorem \hyperref[thm:A]{A} that $\underset{\OFc}{\limn}\left(M\right)=0$
for every $n\ge2$. If, moreover, $\F$ is saturated, then $\underset{\OFc}{\limn[1]}\left(H^{j}\left(-;\Fp\right)\right)=0$
for every $j\ge0$ (see Lemma \ref{lem:vanishing-first-limit}). More
precisely we have the following.

\phantomsection
\label{thm:B}
\begin{namedthm}[Theorem B]
Let $M^{*}$ and $M$ be as in Theorem \hyperref[thm:A]{A}. Assume
that $\F$ is saturated, that $\Fcr_{x}\subseteq\mathcal{C}$ for
every $x\in\left\{ \boldsymbol{1},\boldsymbol{2},\boldsymbol{e}\right\} $
and that $\CGp\left(P\right)=0$ for every $P\in\mathcal{C}$. Then
$\limn_{\OFc}\left(M\right)=0$ for every $n\ge2$ while
\[
\underset{\OFc}{\limn[1]}\left(M\right)\cong M^{\F_{\boldsymbol{e}}}/\left(\overline{\iota_{S'}^{S}}\left(M^{\F_{\boldsymbol{1}}}\right)+M^{\F_{\boldsymbol{2}}}\right).
\]
Moreover $\underset{\OFc}{\limn[i]}\left(H^{j}\left(-;\Fp\right)\right)=0$
for every $i\ge1$ and every $j\ge0$.
\end{namedthm}
As an application of Theorem \hyperref[thm:B]{B}, we prove between
Propositions \ref{prop:shar-Clelland-Parker} and \ref{prop:shar-Parker-Stroth}
our third main result.

\phantomsection
\label{thm:C}
\begin{namedthm}[Theorem C]
Let $\F$ be either a Clelland-Parker fusion system (see \cite{ClellandParker10})
or a Parker-Stroth fusion system (see \cite{ParkerStroth15}), then
$\underset{\OFc}{\limn[i]}\left(H^{j}\left(-;\Fp\right)\right)=0$
for every $i\ge1$ and every $j\ge0$.
\end{namedthm}
Theorem \hyperref[thm:C]{C} complements the work started by Hencke,
Libman and Lynd in \cite[Theorems 1.1 and 5.2]{HLL23}, by not only
reproving that higher limits of the cohomology functor vanish for
the Parker-Stroth fusion systems, but also for all the Clelland-Parker
fusion systems (and not only those whose essential subgroups are all
abelian).

As proven by Aschbacher and Chermak (see \cite{AschCher10}) the Benson-Solomon
fusion systems can also be realized by amalgamated products of $2$
finite groups. However, as noted in \cite[end of Section 4]{BLO06},
the graphs $\tilde{\T}^{P}/C_{G}\left(P\right)$ resulting from these
amalgams are not necessarily trees. Thus Theorem \hyperref[thm:B]{B}
cannot be as directly applied in this case. It is however possible
to use Theorem \hyperref[thm:A]{A} in order to prove the fourth and
final main result of this paper.

\phantomsection
\label{thm:D}
\begin{namedthm}[Theorem D]
Let $\F:=\F_{\Sol}$ be a Benson-Solomon fusion system, let $M$
be as in Theorem \hyperref[thm:A]{A}, let $\theta:\OFc^{\text{op}}\to\Grp$
be the signalizer functor associated to $\F$ (as appearing in \cite[page 941 and Theorem 8.8]{AschCher10})
and define $\vartheta:\OFc^{\text{op}}\to\Grp$ as the composition
$\vartheta\left(P\right)=\Ab\left(\theta\left(P\right)\right)\otimes\Fd$.
The following hold:
\begin{enumerate}
\item For every $n\ge1$ there exists an isomorphism
\[
\Ext_{\Fd\OFc[\F]}^{n}\left(\vartheta\left(P\right),M\right)\cong\limn[n+2]_{\OFc[\F]}\left(M\right).
\]
\item There exists an exact sequence of the form
\[
0\to\underset{\OFc}{\limn[1]}\left(M\right)\to M^{\FAB{S}{B}}/\left(M^{\FAB{S}{H}}+M^{\FAB{S}{K}}\right)\to\underset{\OFc}{\Nat}\left(\vartheta,M\right)\to\underset{\OFc}{\limn[2]}\left(M\right)\to0.
\]
Where $H\cong\Spin_{7}\left(5^{2^{n}}\right)$ for some $n\ge0$ ,
$S\in\Syl_{2}\left(H\right)$ and $B\le K$ are finite groups such
that $\left[K:B\right]=3$ and $S\le B\le H$.
\end{enumerate}
\end{namedthm}
We conclude this introduction with a few words about potential future
work. Let $\F$ be a saturated fusion system over $S$, let $\L$
be its associated centric linking system (see \cite[Definition 1.7]{BLO03}),
let $P_{0}:=S$ and number all the $\F$-essential subgoups of $S$
as $P_{1},\dots,P_{n}$. By defining $L_{i}:=\Aut_{\L}\left(P_{i}\right)$
for every $i$, we obtain that each $L_{i}$ is a finite group with
$P_{i}\trianglelefteq L_{i}$ and $N_{S}\left(P_{i}\right)\in\Syl_{p}\left(L_{i}\right)$.
If, moreover, we define $\F_{i}:=\FAB{N_{S}\left(P_{i}\right)}{L_{i}}$
and $\F_{e,i}:=\FAB{N_{S}\left(P_{i}\right)}{N_{S}\left(P_{i}\right)}$,
we obtain that $\Fcr_{i},\Fcr_{e,i}\subseteq\Fc$ (see Poposition
\ref{prop:centric-radical-contains-normal}). Finally, we can inductively
define $G_{0}:=L_{0}$ and $G_{i}:=G_{i-1}*_{N_{S}\left(P_{i}\right)}L_{i}$,
in order to obtain a group $G:=G_{n}$ such that $\F=\FSG$. Is it
possible to prove the sharpness conjecture for $\F$ by iteratively
applying variations of Theorem \hyperref[thm:B]{B} to each $\FSG[G_{i}]$?
In a future article, in writing, the author explores some cases in
which this question has a positive answer. In so doing he obtains
sharpness for a wider range of fusion systems than that presented
in Theorem \hyperref[thm:C]{C}.

\medskip{}

\hspace*{0.5cm}\textbf{Organization of the paper}: In Section \ref{sec:Preliminaries.}
we briefly introduce the notation used throughout the paper as well
as the concepts of trees of groups and trees of fusion systems. We
also start studying the graphs $\tilde{\T}^{P}/C_{G}\left(P\right)$
and provide an alternative description of the functor $\CGp$ which
motivates our choice of notation (see Proposition \ref{prop:first-homology}
and Remark \ref{rem:choice-of-notation}). We continue in Section
\ref{sec:An-exact-sequence-for=000020higher-limits} where we prove
the core results of this paper (namely Theorems \hyperref[thm:A]{A}
and \hyperref[thm:B]{B}). Finally, Section \ref{sec:sharpness-results}
is dedicated to applying Theorems \hyperref[thm:A]{A} and \hyperref[thm:B]{B}
in order to prove Theorems \hyperref[thm:C]{C} and \hyperref[thm:D]{D}.
\begin{acknowledgement*}
The author would like to thank Chris Parker, Jason Semeraro, Martin
Van Beek, Valentina Grazian, Justin Lynd, Andrew Chermak, Guille Carrión
and Antonio Díaz for their comments and insightful conversations.
These lead to realizing that Theorem \hyperref[thm:B]{B} could be
applied to more cases than originally considered and to great simplifications
of some proofs.

During the development of this article the author was supported by
the Oberwolfach Leibniz Fellows grant, and by the TU Dresden university
via the grants U-000902-830-C11-1010103, F-013038-541-000-1010103
and L-000150-721-110-1010100.
\end{acknowledgement*}

\section{\protect\label{sec:Preliminaries.}Background and first results}

We refer the interested reader to \cite[Part I]{AKO11} for an in
depth introduction to fusion systems. In the present document we limit
ourselves to recalling the following widely adopted notation regarding
fusion systems.
\begin{defn}
\label{def:general-fusion-system-notions}\phantom{.}
\begin{enumerate}
\item Let $S$ be a finite $p$-group and let $G$ be a group containing
$S$. We denote by $\FSG$ the \textbf{fusion system generated by
$G$}. That\textbf{ }is $\FSG$ is the category with objects subgropus
of $S$ and morphisms sets given by conjugation by elements of $G$.
\item \label{enu:realizable-fusion-system}If a fusion system $\F$ satisfies
$\F=\FSG$ for some $G$ finite and $S\in\Syl_{p}\left(G\right)$
we say that $\FSG$ is realizable. Otherwise we say that $\F$ is
exotic.
\item Unless explicitly mentioned, the fusion systems systems introduced
are not necessarily saturated. Note however that Clelland-Parker,
Parker-Stroth and Benson-Solomon fusion systems are all saturated.
\item Let $\F$ be fusion system over $S$, we denote by $\OF$ the \textbf{orbit
category of $\F$}. That is $\OF$ is the category with objects the
subgroups of $S$ (just like $\F$) and morphism sets the quotients
\[
\Hom_{\OF}\left(P,Q\right):=Q\backslash\Hom_{\F}\left(P,Q\right).
\]
\item \label{enu:F-centric}A subgroup $P\le S$ is called \textbf{$\F$-centric
}if $C_{S}\left(Q\right)=Z\left(Q\right)$ for every $Q\cong_{\F}P$
(i.e. every $Q\le S$conjugate to $P$ in $\F$).
\item \label{enu:F-centric-radical}A subgroup $P\le S$ is called \textbf{$\F$-centric-radical
}if it is $\F$-centric and $\Aut_{\F}\left(P\right)=\Inn\left(P\right)$.
\item We denote by $\Fc$ the collection of all $\F$-centric subgroups
of $S$ and by $\Fcr$ the collection of all $\F$-centric-radical
subgroups of $S$.
\item \label{enu:def-centric-orbit-category}For any family $\mathcal{C}$
of groups we denote by $\OFC$ the full subcategory of $\OF$ whose
objects are subgroups of $S$ contained in $\mathcal{C}$. If $\mathcal{C}=\Fc$
we write $\OFc:=\OFC$ and call it the \textbf{centric orbit category
of $\F$}.
\end{enumerate}
Throughout this document we will further adopt the following notation.
\end{defn}

\begin{notation}
\phantom{.}
\begin{enumerate}
\item $p$ is a prime number.
\item For every prime power $q$ we denote by $\Fp[q]$ the finite field
of $q$ elements.
\item Unless otherwise specified all modules are right modules and all actions
are right actions.
\item Given a category $\boldsymbol{C}$ and a commutative ring $\R$ we
use the natural equivalence of categories $\R\boldsymbol{C}\mod{}\cong\R\mod{}^{\boldsymbol{C}^{\text{op}}}$
in order to view all finitely generated $\R\boldsymbol{C}$-modules
as contravariant functors from $\boldsymbol{C}$ to the category $\R\mod{}$
of finitely generated $\R$-modules and viceversa.
\end{enumerate}
\end{notation}

\subsection{\protect\label{subsec:trees-of-fusion-systems}Trees of fusion systems.}

Given a finite $p$-group $S$, it is fairly easy to build a fusion
system over $S$. For instance, given any set $X$ of monomorphisms
between subgroups of $S$, there is an obvious way of building the
smallest fusion system $\left\langle X\right\rangle _{S}$ over $S$
containing all the morphisms in $X$. The question of weather or not
a given fusion system is saturated is however much harder. In fact,
finding and classifying such fusion systems has been a topic of much
research (see for example \cites{ClellandParker10}{DiazRuizViruel07}{GraPar25}{GPSB25}{ParkerSemeraro18}{VanBeek25}).
Some of the most common methods used to determine weather or a fusion
system is saturated makes use of what is known as a tree of groups
(see for instance \cites{BLO06}{ClellandParker10}{GPSB25}{Semeraro14}). 

Let $\T$ be a graph and denote by $V\left(\T\right)$ and $E\left(\T\right)$
its vertices and edge set respectively. Then $\T$ can be viewed as
a category (which we also denote by $\T$) by setting $\Ob\left(\T\right)$
to be the disjoint union of $V\left(\T\right)$ and $E\left(\T\right)$
and letting the only non identity morphisms be morphisms $f_{e,v}:e\to v$
for every $e\in E\left(\T\right)$ and every $v\in V\left(\T\right)$
incident to $e$. That is $\Hom_{\T}\left(e,v\right)=\left\{ f_{e,v}\right\} $
if $v$ is incident to $e$ and is the empty set otherwise. 
\begin{defn}
A \textbf{graph of groups} is a pair $\left(\T,\G\right)$ where $\T$
is a graph and $\G$ is a functor from $\T$ (regarded as a category)
to the category $\Grp$ of groups which sends every morphism to a
group monomorphism. The \textbf{completion of $\left(\T,\G\right)$}
is the group $\G_{\T}=\colim_{\T}\G$. If $\T$ is a tree we call
$\left(\T,\G\right)$ a \textbf{tree of groups}.
\end{defn}

\begin{example}
\label{exa:tree-of-groups}$\phantom{.}$
\begin{enumerate}
\item \label{enu:exa-tree-amalgam}Given groups $B,G_{\boldsymbol{1}},G_{\boldsymbol{2}}$
and monomorphisms $f_{\boldsymbol{i}}:B\hookrightarrow G_{\boldsymbol{i}}$,
we can take the tree $\T:=\boldsymbol{1}\stackrel{\boldsymbol{e}}{--}\boldsymbol{2}$
with only two vertices and a single edge between them and define $\G:\T\to\Grp$
by setting $\G\left(\boldsymbol{i}\right)=G_{\boldsymbol{i}}$, $\G\left(\boldsymbol{e}\right)=B$
and $\G\left(f_{\boldsymbol{e},\boldsymbol{i}}\right)=f_{\boldsymbol{i}}$.
Then $\left(\T,\G\right)$ is a graph of groups and its completion
$\G_{\T}$ is the amalgamated product $\G_{\T}=G_{\boldsymbol{1}}*_{B}G_{\boldsymbol{2}}$
defined by the inclusions $f_{\boldsymbol{i}}$.
\item \label{enu:exa-tree-auto}Let $\F$ be a saturated fusion system over
$S$ and let $\T$ be a spanning tree of the poset of subgroups of
$S$ (ordered by inclusion). For every vertex $P\in V\left(\T\right)$
and every edge $\left(P,Q\right)\in E\left(\T\right)$ with $P\le Q$
define $\G\left(P\right)=P\rtimes\Aut_{\F}\left(P\right)$ and $\G\left(\left(P,Q\right)\right)=P$.
Then $\left(\T,\G\right)$ defines a tree of groups (with the obvious
inclusion morphisms) and $\G_{\T}$ satisfies $\FSG[\G_{\T}]=\F$.
\item \label{enu:exa-tree-linking-system}Taking $\L$ the centric linking
system associated to $\F$ and $\T$ a spanning tree of the poset
of $\F$-centric subgroups of $S$ (or any subposet that contains
all the $\F$-esential subgroups). Then we can define $\G$ by setting
$\G\left(P\right)=\Aut_{\L}\left(P\right)$ and $\G\left(P,Q\right)=\Aut_{\L}\left(P<Q\right)$
(that is the subgroup of $\Aut_{\L}\left(Q\right)$ normalizing $P$).
In this case we also have that $\left(\T,\G\right)$ defines a tree
of groups and $\FSG[\G_{\T}]=\F$.
\end{enumerate}
\end{example}

Using the notation of Example \ref{exa:tree-of-groups}(\ref{enu:exa-tree-linking-system}),
we can take $\T$ to be the tree of height $1$ with root $S$ and
leaves all $\F$-centric and fully $\F$-normalized subgroups of $S$.
Then, for every vertex $P$ of $\T$, $S_{P}:=N_{S}\left(P\right)$
is a Sylow $p$-subgroup of $\G\left(P\right)$ and of $\G\left(P<S\right)$
(if $P\not=S$). Thus, for every vertex $P$ of $\T$, we can define
the saturated fusion systems $\FG\left(P\right):=\FAB{S_{P}}{\G\left(P\right)}$
and, if $P\not=S$, $\FG\left(P<S\right):=\FAB{S_{P}}{\G\left(P<S\right)}$.
With this setup we have that $\FSG[\G_{\T}]=\left\langle \FG\left(P\right)\,|\,P\in V\left(\T\right)\right\rangle _{S}$.
That is $\FSG[\G_{\T}]$ is the smallest fusion system over $S$ containing
the fusion subsystems $\FG\left(P\right)$. A similar result can in
fact be transferred to a more general setting (see Proposition \ref{prop:colimit-fs-exists}
below).

Let $\left(\T,\G\right)$ be a tree of groups and for every $X\in\Ob\left(\T\right)$
choose $\mathcal{S}\left(X\right)\in\Syl_{p}\left(\G\left(X\right)\right)$
and define $\FG\left(X\right):=\F_{\mathcal{S}\left(X\right)}\left(\G\left(X\right)\right)$.
For every edge $e\in E\left(\T\right)$ and every vertex $v$ incident
to $e$ we can choose $g_{e,v}\in\G\left(v\right)$ such that $\G\left(f_{e,v}\right)\circ g_{e,v}$
maps $\mathcal{S}\left(e\right)$ into $\mathcal{S}\left(v\right)$.
This naturally induces a map between the fusion systems $\FG\left(e\right)$
and $\FG\left(v\right)$ which we denote by $\FG\left(f_{e,v}\right)$.
\begin{defn}
We call \textbf{tree of group fusion system arising from $\left(\T,\G\right)$}
(or simply \textbf{tree of group fusion systems}) any pair $\left(\T,\FG\right)$
with $\FG$ a functor, defined as above, mapping $\T$ to the category
of saturated $p$-fusion systems.
\end{defn}

Notice that $\left(\T,\FG\right)$ depends on our choices both of
$\mathcal{S}\left(X\right)$ and $g_{e,v}$. In particular, it is
not necessarily unique. However, under some mild conditions, we can
obtain a uniqueness result.
\begin{defn}
\label{def:p-local-tree-of-groups}We say that a tree of groups $\left(\T,\G\right)$
is \textbf{$p$-local }if there exists $v^{*}\in V\left(\T\right)$
such that, for every $v\in\T\backslash\left\{ v^{*}\right\} $, then
$p\not|\frac{\left|\G\left(v\right)\right|}{\left|\G\left(e\right)\right|}$
where $e$ is the unique edge with extreme $v$ that belongs to the
unique minimal path from $v$ to $v^{*}$. Equivalently $\G\left(f_{e,v}\right)$
maps Sylow $p$-subgroups to Sylow $p$-subgroups.
\end{defn}

\begin{example}
Using the notation of Example \ref{exa:tree-of-groups}(\ref{enu:exa-tree-linking-system}),
we can take $\T$ to be the tree of height $1$ with root $S$ and
leaves all $\F$-centric and fully $\F$-normalized subgroups of $S$.
Then the resulting tree of groups is $p$-local with $v^{*}=S$.

On the other hand, the tree of groups described in Example \ref{exa:tree-of-groups}(\ref{enu:exa-tree-auto})
is not $p$-local unless $p\not|\left|\Aut_{\F}\left(P\right)\right|$
for every $P\lneq S$ (in particular $S$ is abelian).

Finally, with notation as in Example \ref{exa:tree-of-groups}(\ref{enu:exa-tree-amalgam}),
we can see $B$, $G_{\boldsymbol{1}}$ and $G_{\boldsymbol{2}}$ as
subgroups of $\G_{\T}$. Then the tree of groups in this example is
$p$-local if and only if $B$ contains a Sylow $p$-subgroup of either
$G_{\boldsymbol{1}}$ or $G_{\boldsymbol{2}}$.
\end{example}

Let $\left(\T,\G\right)$ be a $p$-local tree of groups, let $\left(\T,\FG\right)$
be a tree of group fusion systems arising from $\left(\T,\G\right)$
and let $S=S\left(v^{*}\right)$. For every vertex $v$ of $\T$,
there exists an obvious monomorphism of $\mathcal{S}\left(v\right)$
into $S$ arising from the unique path from $v$ to $v^{*}$ and the
choices of $g_{e,v}$. This monomorphism allows us to view every $\FG\left(v\right)$
as a fusion system over a subgroup of $S$.
\begin{prop}
\label{prop:colimit-fs-exists}With notation as above, the fusion
system $\FG_{\T}:=\FSG[\G_{\T}]$ is the smallest fusion system over
$S$ containing the fusion systems $\FG\left(v\right)$ for every
vertex $v$ of $\T$. In particular $\FG_{\T}$ does not depend on
the choice of $\left(\T,\FG\right)$.
\end{prop}

\begin{proof}
See \cite[Theorem A(b)]{Semeraro14}.
\end{proof}
It is sometimes possible (see \cite[Theorem 4.2]{BLO06}) to determine
if the fusion system $\FG_{\T}$ is saturated by looking at the orbit
graph associated to $\left(\T,\G\right)$.
\begin{defn}
\label{def:orbit-graph}Let $\left(\T,\G\right)$ be a tree of groups
with completion $G=\G_{\T}$. The \textbf{orbit graph associated to
$\left(\T,\G\right)$} is the graph $\tilde{\T}$ (in fact a tree)
with vertex set the family of right cosets $V\left(\tilde{\T}\right)=\bigsqcup_{v\in V\left(\T\right)}\G\left(v\right)\backslash G$
and for every edge $e:=\left(v,w\right)\in\T$ and every $\G\left(e\right)g\in G$
a single edge connecting $\G\left(v\right)g$ and $\G\left(w\right)g$.
\end{defn}

\begin{lem}
\label{lem:T-as-hocolim}Let $\G\left(-\right)\backslash G:\T\to\Grph$
be the functor sending every $X\in\Ob\left(\T\right)$ to the graph
with no edges and and vertex set the right cosets $\G\left(X\right)\backslash G$.
Then, viewing graphs as $CW$-complexes in the natural way, we have
that $\tilde{\T}=\hocolim_{\T}\left(\G\left(-\right)\backslash G\right)$.
\end{lem}

\begin{proof}
This is immediate from construction of the homotopy colimit.
\end{proof}
The graph $\tilde{\T}$ admits an obvious right $G$-action (as do
the graphs $\G\left(X\right)\backslash G$). This allows us to define,
for every subgroup $P\le G$, the quotient graph $\tilde{\T}^{P}/C_{G}\left(P\right)$.
Here $\tilde{\T}^{P}$ denotes the subgraph of $\tilde{\T}$ of points
fixed under the action of $P$.

Assume that $\left(\T,\G\right)$ is $p$-local, let $S=S\left(v^{*}\right)$
and let $\left(\T,\FG\right)$ be an associated tree of group fusion
systems. We define \label{pag:quotient-graph}$\tilde{\T}^{-}/C_{G}\left(-\right):\mathcal{O}\left(\FG_{\T}\right)^{\text{op}}\to\Grph$
as the  functor that sends every $P\le S$ to the quotient graph $\tilde{\T}^{P}/C_{G}\left(P\right)$.
Several results (see for instance \cite[Theorem 4.2]{BLO06} and \cite[Theorem 3.2]{ClellandParker10}),
require these quotient graphs to be trees whenever $P$ is $\FG_{\T}$-centric
in order to prove that $\FG_{\T}$ is a saturated fusion system.
The multiple applications of such results to find new families of
exotic fusion systems (see \cite{ClellandParker10,ParkerStroth15})
make the graphs $\tilde{\T}^{P}/C_{G}\left(P\right)$ evidently relevant
for the study of saturated fusion systems. The following section is
therefore dedicated to the study of such graphs from the point of
view of trees of group fusion systems.

\subsection{The graphs $\tilde{\T}^{P}/C_{G}\left(P\right)$}

Let $\left(\T,\G\right)$ and $\left(\T,\FG\right)$ be as at the
end of Subsection \ref{subsec:trees-of-fusion-systems} and for every
$X\in\Ob\left(\T\right)$ view $\mathcal{S}\left(X\right)$ as a subgroup
of $S$ as in Proposition \ref{prop:colimit-fs-exists}. Define the
equivalence relation $\sim$ on the set $\Hom_{\F}\left(P,\mathcal{S}\left(X\right)\right)$
by setting $\alpha\sim\beta$ if and only if there exists an isomorphism
$\gamma\in\Hom_{\FG\left(X\right)}\left(\left(P\right)\alpha,\left(P\right)\beta\right)$
such that $\left(x\right)\left(\alpha\circ\gamma\right)=\left(x\right)\beta$
for every $x\in P$. This allows us to define the set 
\[
\Rep_{\F}\left(P,\FG\left(X\right)\right)=\Hom_{\F}\left(P,\mathcal{S}\left(X\right)\right)/\sim,
\]
which we view as a graph with no edges. Given $\alpha\in\Hom_{\F}\left(P,\mathcal{S}\left(X\right)\right)$
we denote by $\left[\alpha\right]_{\F\left(X\right)}$ its equivalence
class in $\Rep_{\F}\left(P,\FG\left(X\right)\right)$.
\begin{defn}
\label{def:RepF-as-hocolim}For every $P\le S$, define $\Rep_{\F}\left(P,\FG\left(-\right)\right)$
as the functor from $\T$ to $\Grph$ sending every $X\in\Ob\left(\T\right)$
to $\Rep_{\F}\left(P,\FG\left(X\right)\right)$. The image of $f_{e,v}$
sends the equivalence class of a morphism $\left[\alpha\right]_{\FG\left(e\right)}$
to $\left[\alpha\circ\iota_{\mathcal{S}\left(e\right)}^{\mathcal{S}\left(v\right)}\right]_{\FG\left(v\right)}$.
By viewing graphs as $CW$-complexes in the natural way, we define
$\Rep_{\F}\left(P,\FG\right)=\hocolim_{\T}\left(\Rep_{\F}\left(P,\FG\left(-\right)\right)\right)$.
The following provides an alternative description of $\Rep_{\F}\left(P,\FG\right)$.
\end{defn}

\begin{lem}
\label{lem:hocolim-as-graph}The $CW$-complex $\Rep_{\F}\left(P,\FG\right)$
can be viewed as a graph $R$ with 
\begin{align*}
V\left(R\right) & :=\bigsqcup_{v\in V\left(\T\right)}\Rep_{\F}\left(P,\FG\left(v\right)\right) & \text{and} &  & E\left(R\right) & :=\bigsqcup_{e\in E\left(\T\right)}\Rep_{\F}\left(P,\FG\left(e\right)\right),
\end{align*}
where, for every edge $e=\left(v,w\right)$ of $\T$, the edge $\left[\alpha\right]_{\FG\left(e\right)}$
connects the vertexes $\left[\alpha\circ\iota_{\mathcal{S}\left(e\right)}^{\mathcal{S}\left(v\right)}\right]_{\FG\left(v\right)}$
and $\left[\alpha\circ\iota_{\mathcal{S}\left(e\right)}^{\mathcal{S}\left(w\right)}\right]_{\FG\left(w\right)}$.
\end{lem}

\begin{proof}
This is precisely what is proven in \cite[Lemma 2.8]{Semeraro14}
(see also the comments in \cite[page 1060]{Semeraro14}).
\end{proof}
From now on and unless otherwise specified we make use of Lemma \ref{lem:hocolim-as-graph}
in order to view $\Rep_{\F}\left(P,\FG\right)$ as a graph.
\begin{defn}
We define $\Rep_{\F}\left(-,\FG\right):\OF^{\text{op}}\to\Grph$ as
the functor sending every $P\le S$ to the (potentially empty) graph
$\Rep_{\F}\left(P,\FG\right)$ and every $\left[\varphi\right]\in\Hom_{\OF}\left(Q,P\right)$
to the graph morphism sending the vertex $\left[\alpha\right]_{\FG\left(X\right)}$
of $\Rep_{\F}\left(P,\FG\right)$ to the vertex $\left[\alpha\circ\varphi\right]_{\FG\left(X\right)}$
of $\Rep_{\F}\left(Q,\FG\right)$.
\end{defn}

The introduction of the functor $\Rep_{\F}\left(-,\FG\right)$ provides
us with new insight regarding the orbit graph $\tilde{\T}$. More
precisely we have the following.
\begin{lem}
\label{lem:isomorphism-quotient-graph}Let $G=\G_{\T}$ be the completion
of $\left(\T,\G\right)$. The functor $\Rep_{\F}\left(-,\FG\right)$
is isomorphic to the functor $\tilde{\T}^{-}/C_{G}\left(-\right)$
defined in Page \pageref{pag:quotient-graph}. In particular the functor
$\Rep_{\F}\left(-,\FG\right)$ does not (up to isomorphism) depend
on the choice of $\left(\T,\FG\right)$.
\end{lem}

\begin{proof}
Since $\left(\T,\G\right)$ is $p$-local then it satisfies what in
\cite{Semeraro14} is denoted as Property (H). The result follows
from the natural isomorphims described in \cite[Proposition 3.5 and Corollary 3.7]{Semeraro14}.
\end{proof}
\medskip{}

Let $\R$ be a commutative ring. By viewing graphs as $CW$-complexes
in the natural way, we can compose the functor $\Rep_{\F}\left(-,\FG\right)$
with the $n$-th homology functor $H_{n}\left(-;\R\right)$ with coefficients
in $\R$. We denote such composition as $H_{n}\left(\Rep_{\F}\left(-,\FG\right);\R\right)$.
As an immediate corollary of Lemma \ref{lem:isomorphism-quotient-graph},
we obtain the following.
\begin{cor}
\label{cor:homology-above-2-is-0}The functor $H_{n}\left(\Rep_{\F}\left(-,\FG\right);\R\right)$
is the $0$ functor for every $n\ge2$ while, for $n=0$, it is the
constant contravariant functor $\underline{\R}$ sending every object
to the trivial $\R$-module $\R$.
\end{cor}

\begin{proof}
Let $P\le S$ and view $\Rep_{\F}\left(P,\FG\right)$ as a $CW$-complex.
Since $\Rep_{\F}\left(P,\FG\right)$ has no $n$-cells for every $n\ge2$,
then $H_{n}\left(\Rep_{\F}\left(-,\FG\right);\R\right)=0$ for every
such $n$.

On the other hand we have that $H_{0}\left(\Rep_{\F}\left(P,\FG\right);\R\right)=\R^{k}$,
where $k$ is the number of connected components of $\Rep_{\F}\left(P,\FG\right)$.

We know from \cite[Theorem I.9]{Serre80} (where $\tilde{\T}$ is
denoted by $X$) that $\tilde{\T}$ is a tree on which $G$ acts without
inversion. It follows from \cite[Section I.6.1]{Serre80} that $\tilde{\T}^{P}$
is also a tree. In particular $\tilde{\T}^{P}$ (and therefore $\tilde{\T}^{P}/C_{G}\left(P\right)$)
is connected. The result follows from Lemma \ref{lem:isomorphism-quotient-graph}.
\end{proof}
Corollary \ref{cor:homology-above-2-is-0} gives a description of
the functor $H_{n}\left(\Rep_{\F}\left(-,\FG\right);\R\right)$ for
every $n\not=1$. For this last case we introduce the following. 
\begin{defn}
\label{def:CGp}We denote the functor $H_{1}\left(\Rep_{\F}\left(-,\FG\right);\Fp\right)$
simply as $\CGp$.
\end{defn}

In the cases described at the end of Subsection \ref{subsec:trees-of-fusion-systems}
we have that $\tilde{\T}^{P}/C_{G}\left(P\right)$ is a tree and,
therefore, $\CGp\left(P\right)=0$. However, this functor is not the
$0$-functor and the following characterization can prove useful to
better understand it.
\begin{prop}
\label{prop:first-homology}If $P\in\Fc$, then $\CGp\left(P\right)\cong\Ab\left(C_{G}\left(P\right)/Z\left(P\right)\right)\otimes\Fp$
where $\Ab\left(-\right)$ denotes the abelianization functor (i.e.
the functor sending a group to it's quotient by the commutator subgroup).
\end{prop}

\begin{proof}
Throughout this proof, for every subgroup $H\le G$, we write $C_{H}:=C_{H}\left(P\right)$
and, if $Z\left(P\right)\le H$, we also write $\overline{C_{H}}:=C_{H}/Z\left(P\right)$.

By defining $g\cdot x=x\cdot g^{-1}$ for every $x\in\tilde{\T}^{P}$
and every $g\in C_{G}$ we obtain a left action of $C_{G}$ on $\tilde{\T}^{P}$
such that the quotients $C_{G}\backslash\tilde{\T}^{P}$ and $\tilde{\T}^{P}/C_{G}$
of $\tilde{\T}^{P}$ via both the left and right action of $C_{G}$
coincide. For the reminder of the proof we view $C_{G}$ as acting
on the left on $\tilde{\T}^{P}$.

Since, by definition, $Z\left(P\right)$ acts trivially on $\tilde{\T}^{P}$,
then the action of $C_{G}$ on $\tilde{\T}^{P}$ induces an action
of $\overline{C_{G}}$ on $\tilde{\T}^{P}$ such that the quotients
$C_{G}\backslash\tilde{\T}^{P}$ and $\overline{C_{G}}\backslash\tilde{\T}^{P}$
coincide. We deduce from Lemma \ref{lem:isomorphism-quotient-graph}
that $\CGp\left(P\right)\cong H_{1}\left(\overline{C_{G}}\backslash\tilde{\T}^{P};\Fp,\right)$.
We know from Corollary \ref{cor:homology-above-2-is-0} that $H_{0}\left(\overline{C_{G}}\backslash\tilde{\T}^{P};\mathbb{Z}\right)\cong\mathbb{Z}$
is a free $\mathbb{Z}$-module. It follows that $\Tor\left(H_{0}\left(\overline{C_{G}}\backslash\tilde{\T}^{P};\mathbb{Z}\right),\Fp\right)=0$
(see \cite[Proposition 3A.5(3)]{Hatcher02}). We conclude from the
universal coefficients theorem (see \cite[Theorem 3A.3]{Hatcher02})
that $\CGp\left(P\right)\cong H_{1}\left(\overline{C_{G}}\backslash\tilde{\T}^{P};\mathbb{Z}\right)\otimes\Fp$.
Therefore, it suffices to prove that $\Ab\left(\pi_{1}\left(\overline{C_{G}}\backslash\tilde{\T}^{P}\right)\right)\otimes\Fp\cong\Ab\left(\overline{C_{G}}\right)\otimes\Fp$.

\medskip{}

Proceeding as in \cite[Definition 3.2]{Bass93}, we can take a graph
of groups $\overline{C_{G}}\backslash\backslash\tilde{\T}^{P}:=\left(\overline{C_{G}}\backslash\tilde{\T}^{P},\mathcal{A}\right)$
with the functor $\mathcal{A}$ satisfying that, for every $\overline{X}\in\Ob\left(\overline{C_{G}}\backslash\tilde{\T}^{P}\right)$,
there exists $X\in\Ob\left(\tilde{\T}^{P}\right)$ such that $X$
projects to $\overline{X}$ via the quotient map and $\mathcal{A}\left(\overline{X}\right)=C_{\overline{C_{G}}}\left(X\right)$.

From \cite[Example (3) page 13]{Bass93} and \cite[Proposition I.20]{Serre80}
(see also \cite[Proposition 1.20]{Bass93}), we know that there exists
a group $\pi_{1}\left(\overline{C_{G}}\backslash\backslash\tilde{\T}^{P}\right)$
that contains every $\mathcal{A}\left(\overline{v}\right)$ with $\overline{v}$
a vertex of $\overline{C_{G}}\backslash\tilde{\T}^{P}$ and that fits
in a short exact sequence of the form
\begin{equation}
1\to N\to\pi_{1}\left(\overline{C_{G}}\backslash\backslash\tilde{\T}^{P}\right)\to\pi_{1}\left(\overline{C_{G}}\backslash\tilde{\T}^{P}\right)\to1,\label{eq:exa-seq-for-homology}
\end{equation}
where $N$ is the normal subgroup generated by all the $\mathcal{A}\left(\overline{v}\right)$.
On the other hand, from \cite[Theorem 3.6(d)]{Bass93}, we also have
a short exact sequence of the form
\[
1\to\pi_{1}\left(\tilde{\T}^{P}\right)\to\pi_{1}\left(\overline{C_{G}}\backslash\backslash\tilde{\T}^{P}\right)\to\overline{C_{G}}\to1.
\]

Since $\tilde{\T}^{P}$ is a tree (see \cite[Section I.6.1]{Serre80}),
we deduce that $\pi_{1}\left(\overline{C_{G}}\backslash\backslash\tilde{\T}^{P}\right)\cong\overline{C_{G}}$.
Therefore, using the exact sequence of Equation (\ref{eq:exa-seq-for-homology}),
we obtain an epimorphism 
\begin{equation}
\Ab\left(\overline{C_{G}}\right)\twoheadrightarrow\Ab\left(\pi_{1}\left(\overline{C_{G}}\backslash\tilde{\T}^{P}\right)\right),\label{eq:epimosphism-abelianizations}
\end{equation}
 whose kernel is generated by the images in $\Ab\left(\overline{C_{G}}\right)$
of groups of the form $\mathcal{A}\left(\overline{v}\right)=C_{\overline{C_{G}}}\left(v\right)$
for some vertex $v$ of $\tilde{\T}^{P}$. Fix one such vertex $v$.
By definition of $\tilde{\T}$, there exists $Y\in\Ob\left(\T\right)$
and $y\in G$ such that $v=\G\left(Y\right)y$. It follows that, for
every $x\in C_{C_{G}}\left(v\right)\cup P$, we have $\G\left(Y\right)yx^{-1}=\G\left(Y\right)y$
and, therefore, $x\in\G\left(Y\right)^{y}$. We conclude that $P$
(and in particular $Z\left(P\right)$) is contained in $\G\left(Y\right)^{y}$
and that $\mathcal{A}\left(\overline{v}\right)\le\overline{C_{\G\left(Y\right)^{y}}}$.
Since $P\in\Fc$, then $Z\left(P\right)$ is a Sylow $p$-subgroup
of $C_{H}$ for any $P\le H\le G$ finite. It follows that $\overline{C_{\G\left(Y\right)^{y}}}$
(and therefore $\mathcal{A}\left(\overline{v}\right)$) is a finite
$p'$-group. We deduce that the subgroup of $\Ab\left(\overline{C_{G}}\right)$
obtained as image of $\mathcal{A}\left(\overline{v}\right)$ vanishes
after tensoring by $\Fp$. We conclude that the epimorphism of Equation
(\ref{eq:epimosphism-abelianizations}) induces an isomorphism from
$\Ab\left(\overline{C_{G}}\right)\otimes\Fp$ to $\Ab\left(\pi_{1}\left(\overline{C_{G}}\backslash\tilde{\T}^{P}\right)\right)\otimes\Fp$
thus concluding the proof.
\end{proof}
\begin{rem}
\label{rem:choice-of-notation}Proposition \ref{prop:first-homology}
is in fact the motivation of the chosen notation of $\CGp$ which
aims to evidence that first we focus on the $p'$-part of $C_{G}\left(P\right)$
by quotienting out by it's $p$-part (namely $Z\left(P\right)$) and
then we remove all elements of finite order $p'$ by tensoring with
$\Fp$.
\end{rem}

\begin{cor}
Let $P\in\Fc$ and let $v$ be a vertex of $\T$ such that $P\le\mathcal{S}\left(v\right)$.
If $C_{G}\left(P\right)$ is conjugate to a subgroup of $\G\left(v\right)$,
then $\CGp\left(P\right)=0$.
\end{cor}

\begin{proof}
By taking conjugates of $P$ if necessary, we can assume without loss
of generality that $C_{G}\left(P\right)\le\G\left(v\right)$. In particular
$C_{G}\left(P\right)=C_{\G\left(v\right)}\left(P\right)$. Since $P$
is $\FG\left(v\right)$-centric (because it is $\F$-centric), then
$Z\left(P\right)\in\Syl_{p}\left(C_{G}\left(P\right)\right)$. It
follows that $\Ab\left(C_{G}\left(P\right)/Z\left(P\right)\right)$
is a $p'$-group. In particular $\Ab\left(C_{G}\left(P\right)/Z\left(P\right)\right)\otimes\Fp=0$.
The result follows from Proposition \ref{prop:first-homology}.
\end{proof}
We conclude this section with some observations regarding the fusion
Bredon cohomology of the orbit graph $\tilde{\T}$.

Let $K$ be a $G$-$CW$-complex, let $\R$ be a commutative ring
and let $\mathcal{C}$ be any family of subgroups of $S$. For every
$P\in\mathcal{C}$, define $C_{*}\left(K^{P}/C_{G}\left(P\right);\R\right)$
as the $\R$-module chain complex associated to the $CW$-complex
$K^{P}/C_{G}\left(P\right)$ (i.e. $C_{n}\left(K^{P}/C_{G}\left(P\right);\R\right)$
is the free $\R$-module with basis the $n$-cells of $K^{P}/C_{G}\left(P\right)$).
This construction leads naturally to the description of a chain complex
$C_{*}\left(K^{-}/C_{G}\left(-\right);\R\right)$ of $\R\OFC$-modules.
Let $P_{*,*}$ be a Cartan-Eilenberg projective resolution of $C_{*}\left(K^{-}/C_{G}\left(-\right);\R\right)$
(see \cite[Section 10.5]{Rot79} for a definition). Then, for every
$\R\OFC$-module $M$, we define the following.
\begin{defn}[{\cite[Definition 1.5]{Yal22}}]
\label{def:fusion-Bredon-cohomology}The \textbf{fusion Bredon cohomology}
of $X$ with coefficients in $M$ is defined as
\[
H_{\OFC}^{*}\left(K;M\right):=H^{*}\left(\Hom_{\R\OFC}\left(\Tot^{\oplus}\left(P_{*,*}\right);M\right)\right).
\]
\end{defn}

If $M$ comes from the contravariant part of a Mackey functor with
coefficients in $\Fp$, all the elements in $\mathcal{C}$ are $\F$-centric,
$\mathcal{C}$ is closed under $\F$-conjugation and taking overgroups,
and contains all the $\FG\left(X\right)$-centric-radical groups for
every $X\in\Ob\left(\T\right)$, then it follows from \cite[Proposition 10.5]{Yal22}
and sharpness (see \cite[Theorem B]{DiazPark15}) that $\limn_{\OFC[\FG\left(X\right)]}\left(M\downarrow_{\OFC[\FG\left(X\right)]}^{\OFC}\right)=0$
for every $n\ge1$. In situations such as these, the following simplification
of the Bredon cohomology is possible.
\begin{prop}
\label{prop:cohomology-iso-higher-limit}Assume that $\limn_{\OFC[\FG\left(X\right)]}\left(M\downarrow_{\OFC[\FG\left(X\right)]}^{\OFC}\right)=0$
for every $n\ge1$ and every $X\in\Ob\left(\T\right)$. Then viewing
the $G$-graph $\tilde{\T}$ as a $G$-$CW$-complex in the natural
way we have that
\[
H_{\OFC}^{n}\left(\tilde{\T};M\right)\cong\limn_{\T}\left(\lim_{\OFC[\FG\left(-\right)]}\left(M\downarrow_{\OFC[\FG\left(-\right)]}^{\OFC}\right)\right).
\]
\end{prop}

\begin{proof}
We know from Lemma \ref{lem:T-as-hocolim} that $\tilde{\T}=\hocolim_{\T}\left(\G\left(-\right)\backslash G\right)$.
Therefore, we obtain the following Bousfield-Kahn spectral sequence
(see \cite[\S XII 4.5]{BousKan72})
\begin{equation}
E_{2}^{s,t}:=\limn[s]_{\T}\left(H_{\OFC}^{t}\left(\G\left(-\right)\backslash G;M\right)\right)\Rightarrow H_{\OFC}^{s+t}\left(\tilde{\T};M\right).\label{eq:bousfield-kahn-spectral-seq}
\end{equation}

Let $X\in\mathcal{C}$ and, for every $P\in\mathcal{C}$, define $\Rep_{G}\left(P,\G\left(X\right)\right):=\left(\G\left(X\right)\backslash G\right)^{P}/C_{G}\left(P\right)$.
By applying \cite[Proposition 5.7]{Yal22} on the $G$-$CW$ complex
$\G\left(X\right)\backslash G$, we obtain the converging spectral
sequence
\begin{equation}
\lui{II}{E_{2}^{u,v}}:=\Ext_{\R\OFC}^{u}\left(H_{v}\left(\Rep_{G}\left(-,\G\left(X\right)\right);\R\right),M\right)\Rightarrow H_{\OFC}^{u+v}\left(\G\left(X\right)\backslash G;M\right).\label{eq:sharp-spectral-sequence}
\end{equation}
Since $\G\left(X\right)\backslash G$ is (as a $CW$-complex) just
a discrete set of points, then the same holds for $\Rep_{G}\left(P,\G\left(X\right)\right)$
for every $P\in\mathcal{C}$. It follows that $H_{v}\left(\Rep_{G}\left(-,\G\left(X\right)\right);\R\right)=0$
for every $v\ge1$. In other words the spectral sequence of Equation
(\ref{eq:sharp-spectral-sequence}) is sharp. We deduce that 
\[
\Ext_{\R\OFC}^{t}\left(H_{0}\left(\Rep_{G}\left(-,\G\left(X\right)\right);\R\right),M\right)\cong H_{\OFC}^{t}\left(\G\left(X\right)\backslash G;M\right).
\]
We know from \cite[page 266]{Symonds05} (see also \cite[Proposition 4.1]{Yal22}),
that the functor $H_{0}\left(\Rep_{G}\left(-,\G\left(X\right)\right);\R\right)$
is in fact the induction from $\OFC[\FG\left(X\right)]$ to $\OFC$
of the constant functor $\underline{\R}$ sending everything to the
trivial $\R$-module $\R$. It follows from Shapiro's isomorphism
(see \cite[Proposition 4.8]{Yal22} for a version on fusion systems),
that 
\[
H_{\OFC}^{t}\left(\G\left(X\right)\backslash G;M\right)\cong\limn[t]_{\OFC[\FG\left(X\right)]}\left(M\downarrow_{\OFC[\FG\left(X\right)]}^{\OFC}\right),
\]
where we are using the natural isomorphism $\Ext_{\R\OFC[\FG\left(X\right)]}^{t}\left(\underline{\R},M\right)\cong\limn[t]_{\OFC[\FG\left(X\right)]}\left(M\right)$.
We deduce from the assumptions that the spectral sequence of Equation
(\ref{eq:bousfield-kahn-spectral-seq}) is sharp. The result follows.
\end{proof}

\section{\protect\label{sec:An-exact-sequence-for=000020higher-limits}An
exact sequence for higher limits}

From now onward we focus our attention on the simplest non trivial
tree of fusion systems. To this end we introduce the following notation.
\begin{notation}
\label{nota:simple-tree-of-fusion-systems}$\phantom{.}$
\begin{enumerate}
\item Define $\T$ as the tree with only two vertexes $\left\{ \boldsymbol{1},\boldsymbol{2}\right\} $
and a single edge $\boldsymbol{e}$ between them. When viewing $\T$
as a category, for $i=1,2$, we denote by $f_{\boldsymbol{ei}}$ the
unique morphism between $\boldsymbol{e}$ and $\boldsymbol{i}$.
\item Let $G_{\boldsymbol{1}},G_{\boldsymbol{2}},G_{\boldsymbol{e}}$ be
finite groups with monomorphisms $\varphi_{\boldsymbol{i}}:G_{\boldsymbol{e}}\hookrightarrow G_{\boldsymbol{i}}$
such that $p\not|\left[G_{\boldsymbol{2}}:\varphi_{\boldsymbol{2}}\left(G_{\boldsymbol{e}}\right)\right]$.
In particular a Sylow $p$-subgroup of $G_{\boldsymbol{e}}$ maps
under $\boldsymbol{\varphi_{2}}$ to a Sylow $p$-subgroup of $G_{\boldsymbol{2}}$.
The same need not be true for $i=1$.
\item Define $\G:\T\to\Grp$ as the functor sending $x\in\Ob\left(\T\right)$
to $\G\left(x\right):=G_{x}$ and $f_{\boldsymbol{ei}}$ to $\G\left(f_{\boldsymbol{ei}}\right):=\varphi_{\boldsymbol{i}}:G_{\boldsymbol{e}}\hookrightarrow G_{\boldsymbol{i}}$.
\item Define $G=G_{\boldsymbol{1}}*_{G_{\boldsymbol{e}}}G_{\boldsymbol{2}}$
as the amalgam induced by the monomorphisms $\varphi_{\boldsymbol{i}}$.
In other words, $G$ is the colimit of $\G$ over $\T$.
\item View $G_{\boldsymbol{1}},G_{\boldsymbol{2}}$ and $G_{\boldsymbol{e}}$
as subgroups of $G$ in the natural way. In particular $G_{\boldsymbol{e}}\le G_{\boldsymbol{1}},G_{\boldsymbol{2}}$.
\item Let $S_{\boldsymbol{1}}:=S\in\Syl_{p}\left(G_{\boldsymbol{1}}\right)$
and $S_{\boldsymbol{2}}:=S_{\boldsymbol{e}}:=S'\in\Syl_{p}\left(G_{\boldsymbol{e}}\right)\subseteq\Syl_{p}\left(G_{\boldsymbol{2}}\right)$.
\item Define $\FG$ as the functor from $\T$ to the category of fusion
systems that sends every $x\in\Ob\left(\T\right)$ to $\FG\left(x\right):=\FAB{S_{x}}{G_{x}}$.
\item Define $\F_{x}:=\FG\left(x\right)$ for every $x\in\Ob\left(\T\right)$.
\end{enumerate}
\end{notation}

With this setup we can now prove the following particular case of
Proposition \ref{prop:cohomology-iso-higher-limit}.
\begin{lem}
\label{lem:vanishing-cohomology-higher-than-2}Let $\mathcal{C}$
be a family of subgroups of $S$, let $\R$ be a commutative ring
and let $M$ be an $\R\OFC$-module. Assume that $\limn_{\OFC[\F_{x}]}\left(M\downarrow_{\OFC[\F_{x}]}^{\OFC}\right)=0$
for every $n\ge1$ and every $x\in\Ob\left(\T\right)$. Then
\[
H_{\OFC}^{n}\left(\tilde{\T};M\right)\cong\begin{cases}
M^{\F} & \text{if }n=0\\
M^{\F_{\boldsymbol{e}}}/\left(\overline{\iota_{S'}^{S}}\left(M^{\F_{\boldsymbol{1}}}\right)+M^{\F_{\boldsymbol{2}}}\right) & \text{if }n=1\\
0 & \text{else}
\end{cases},
\]
where $M^{\mathcal{H}}:=\lim_{\OFC[\mathcal{H}]}M\downarrow_{\OFC[\mathcal{H}]}^{\OFC}$
is viewed as a subgroup of $M\left(S\right)$ or $M\left(S'\right)$
as appropriate and $\overline{\iota_{S'}^{S}}:M^{\F_{\boldsymbol{1}}}\to M^{\F_{\boldsymbol{e}}}$
is the universal map.
\end{lem}

\begin{proof}
Define $P_{0}$ and $P_{1}$ as the contravariant functors from $\T$
to $\R$-mod satisfying
\begin{align*}
P_{0}\left(\boldsymbol{e}\right) & =\R^{2}, & P_{0}\left(\boldsymbol{1}\right) & =P_{0}\left(\boldsymbol{2}\right)=\R,\\
P_{1}\left(\boldsymbol{e}\right) & =\R, & P_{1}\left(\boldsymbol{1}\right) & =P_{1}\left(\boldsymbol{2}\right)=0,
\end{align*}
and with $P_{0}\left(f_{\boldsymbol{e},\boldsymbol{i}}\right)$ the
natural inclusion on the $i$-th component.

Let $\underline{\R}$ be the constant contravariant functor from $\T$
to $\R\mod{}$ which sends everything to $\R$. Define the natural
transformations $\varepsilon:P_{0}\to\underline{\R}$ and $\varphi:P_{1}\to P_{0}$
by setting
\begin{align*}
\varepsilon_{\boldsymbol{i}}\left(x\right) & =x, & \varphi_{\boldsymbol{e}}\left(x\oplus y\right)=x+y & ,\\
\varepsilon_{\boldsymbol{e}}\left(x\right) & =x\oplus-x.
\end{align*}

Then $0\to P_{1}\stackrel{\varphi}{\to}P_{0}\stackrel{\varepsilon}{\to}\underline{\R}$
is a projective resolution of $\underline{\R}$. The result follows
from Proposition \ref{prop:cohomology-iso-higher-limit} by using
this projective resolution to compute the higher limits.
\end{proof}
We are now ready to proof the following.
\begin{prop}
\label{prop:long-exact-sequence}Let $\mathcal{C}$ be a family of
subgroups of $S$, let $\R$ be a commutative ring and let $M$ be
an $\R\OFC$-module. Assume that $\limn_{\OFC[\F_{x}]}\left(M\downarrow_{\OFC[\F_{x}]}^{\OFC}\right)=0$
for every $n\ge1$ and every $x\in\Ob\left(\T\right)$. Then the following
hold:
\begin{enumerate}
\item For every $n\ge1$ there exists an isomorphism
\[
\Ext_{\R\OFC}^{n}\left(\CGp,M\right)\cong\limn[n+2]_{\OFC}\left(M\right).
\]
\item There exists an exact sequence of the form
\[
0\to\underset{\OFC}{\limn[1]}\left(M\right)\to M^{\F_{\boldsymbol{e}}}/\left(\overline{\iota_{S'}^{S}}\left(M^{\F_{\boldsymbol{1}}}\right)+M^{\F_{\boldsymbol{2}}}\right)\to\underset{\R\OFC}{\Hom}\left(\CGp,M\right)\to\underset{\OFC}{\limn[2]}\left(M\right)\to0.
\]
\end{enumerate}
\end{prop}

\begin{proof}
Let $\tilde{\mathcal{T}}$ be the orbit graph associated with $\left(\G,\T\right)$
(see Definition \ref{def:orbit-graph}). By viewing graphs as $CW$-complexes
in the natural way, we can apply \cite[Proposition 5.7]{Yal22} with
$X=\tilde{\T}$, in order to obtain the converging spectral sequence
\[
\lui{II}{E_{2}^{s,t}}:=\Ext_{\R\OFC}^{s}\left(H_{t}\left(\Rep_{\F}\left(-,\FG\right)\right),M\right)\Rightarrow H^{s+t}:=H_{\OFC}^{s+t}\left(\tilde{\T};M\right),
\]
here we are applying Lemma \ref{lem:isomorphism-quotient-graph} in
order to replace the functor $\tilde{\T}^{-}/C_{G}\left(-\right)$
with the functor $\Rep_{\F}\left(-,\FG\right)$ (or rather its restriction
to $\OFC$).

We know from Corollary \ref{cor:homology-above-2-is-0}, that $\lui{II}{E_{2}^{s,0}}\cong\Ext_{\Fp\OFc}^{s}\left(\underline{\Fp},M\right)\cong\limn[s]_{\OFc}\left(M\right)$
and that $\lui{II}{E_{2}^{s,t}}=0$ for every $t\ge2$. Using the
above spectral sequence, we deduce (see \cite[Exercise 5.2.2]{Weibel94})
that there exists a long exact sequence of the form
\begin{equation}
\cdots\to H^{n}\to\Ext_{\R\OFc}^{n-1}\left(\CGp,M\right)\to\limn[n+1]_{\OFc}\left(M\right)\to H^{n+1}\to\cdots.\label{eq:long-exact-sequence}
\end{equation}
 The result follows from Lemma \ref{lem:vanishing-cohomology-higher-than-2}.
\end{proof}
\begin{rem}
In the proof of Proposition \ref{prop:long-exact-sequence}, we have
used the fact that $\T$ is the tree with two vertices only in order
to prove that the fusion Bredon cohomology $H_{\OFC}^{n}\left(\tilde{\T};M\right)$
vanishes for $n\ge2$. In the more general setting where $\left(\T,\G\right)$
is any $p$-local tree of groups, we know that, if $\limn_{\OFC[\FG\left(X\right)]}\left(M\downarrow_{\OFC[\FG\left(X\right)]}^{\OFC}\right)=0$
for every $X\in\Ob\left(\T\right)$ and every $n\ge1$, this Bredon
cohomology can be seen as the higher limits $\limn_{\T}\left(\lim_{\OFC[\FG\left(-\right)]}\left(M\downarrow_{\OFC[\FG\left(-\right)]}^{\OFC}\right)\right)$
(see Proposition \ref{prop:cohomology-iso-higher-limit}). Thus, in
this more general setting, we still obtain the long exact sequence
of Equation (\ref{eq:long-exact-sequence}) but further simplifications
of $H^{n}$ are not covered in this work.

As an immediate consequence of Proposition \ref{prop:long-exact-sequence}
we can now provide a proof of Theorem \hyperref[thm:A]{A}.
\end{rem}

\begin{proof}[proof of Theorem A]
We know from \cite[Proposition 10.5]{Yal22} (see also \cite[Corollary 3.6]{BLO03})
that for every $\mathcal{H}\in\left\{ \F,\F_{\boldsymbol{1}},\F_{\boldsymbol{2}},\F_{\boldsymbol{e}}\right\} $
and every $n\ge0$
\[
\limn_{\OFc[\mathcal{H}]}\left(M^{*}\downarrow_{\OFc}^{\OF}\right)\cong\limn_{\OFC[\mathcal{H}]}\left(M^{*}\downarrow_{\OFC[\mathcal{H}]}^{\OF}\right).
\]
We conclude from \cite[Theorem B]{DiazPark15} that $\limn_{\OFc[\mathcal{E}]}\left(M^{*}\downarrow_{\OFc}^{\OF}\right)=0$
for every $n\ge1$ and every $\mathcal{E}\in\left\{ \F_{\boldsymbol{1}},\F_{\boldsymbol{2}},\F_{\boldsymbol{e}}\right\} $.
The result follows from Proposition \ref{prop:long-exact-sequence}.
\end{proof}
Theorem \hyperref[thm:A]{A} tells us that, in the cases where $\CGp\left(P\right)=0$
for every $P\in\mathcal{C}$ then $\limn[i]_{\OFc}\left(H^{j}\left(-;\Fp\right)\right)=0$
for every $i\ge2$ and every $j\ge0$. Thus, in order to prove Theorem
\hyperref[thm:B]{B} we just need to prove that the first higher limit
also vanishes. To this end we first recall the following well known
result.
\begin{lem}[{\cite[Exercise 5.2.1 (dual)]{Weibel94}}]
\label{lem:short-exact-sequence}Let $\boldsymbol{A}$ be an abelian
category and let $E_{2}^{s,t}\Rightarrow H^{s+t}$ be a cohomology
spectral sequence in $\boldsymbol{A}$. Assume that $E_{2}^{s,t}=0$
for every $s\not\in\left\{ 0,1\right\} $. For every $n\in\mathbb{Z}$,
there exists a short exact sequence of the form
\[
0\to E_{2}^{1,n-1}\to H^{n}\to E_{2}^{0,n}\to0.
\]
\end{lem}

As a consequence of Lemma \ref{lem:short-exact-sequence} we obtain
the following.
\begin{lem}
\label{lem:vanishing-first-limit}Let $\F$ be a saturated fusion
system. Assume that $\limn[i]_{\OFc}\left(H^{j}\left(-;\Fp\right)\right)=0$
for every $j\ge0$ and every $i\ge2$. Then $\limn[1]_{\OFc}\left(H^{j}\left(-;\Fp\right)\right)=0$
for every $j\ge0$.
\end{lem}

\begin{proof}
Since $\F$ is saturated then we can take the centric linking system
$\L$ of $\F$ and define $B\F:=\left|\L\right|_{p}^{\mathcircumflex}$.
From \cite[Proposition 2.2]{BLO03}, there exists a Bousfield-Kan
spectral sequence of the form
\[
\limn[i]_{\OFc}\left(H^{j}\left(-;\Fp\right)\right)\Rightarrow H^{i+j}\left(B\F;\Fp\right).
\]

From Lemma \ref{lem:short-exact-sequence}, we obtain, for every $n\ge1$
a short exact sequence of the form 
\[
0\to\limn[1]_{\OFc}\left(H^{n-1}\left(-;\Fp\right)\right)\to H^{n}\left(B\F;\Fp\right)\to\lim_{\OFc}\left(H^{n}\left(-;\Fp\right)\right)\to0.
\]

From \cite[Lemma 5.3]{BLO03}, we know that $\lim_{\OFc}\left(H^{n}\left(-;\Fp\right)\right)\cong H^{n}\left(B\F;\Fp\right)$
for every $n\ge0$. Since $\lim_{\OFc}\left(H^{n}\left(-;\Fp\right)\right)$
is a finitely generated $\Fp$-module, and therefore finite, we deduce
that the epimorphism of the above short exact sequence is, in fact,
an isomorphism. The result follows.
\end{proof}
Theorem \hyperref[thm:B]{B} is now an immediate consequence of Theorem
\hyperref[thm:A]{A} and Lemma \ref{lem:vanishing-first-limit}.

\section{\protect\label{sec:sharpness-results}Sharpness for some fusion systems
coming from amalgams}

This section is dedicated to applying Theorems \hyperref[thm:A]{A}
and \hyperref[thm:B]{B} to the study of three families of fusion
systems. Namely we use Theorem \hyperref[thm:B]{B} in order to prove
Theorem \hyperref[thm:C]{C} (see Propositions \ref{prop:shar-Clelland-Parker}
and \ref{prop:shar-Parker-Stroth}) thus showing that sharpness holds
for both the Clelland-Parker and the Parker-Stroth fusion systems.
Moreover we use Theorem \hyperref[thm:A]{A} to prove Theorem \hyperref[thm:D]{D}
(see Subsection \ref{subsec:Benson-Solomon}) thus relating higher
limits over the orbit category of Benson-Solomon fusion systems with
the signalizer functor of these fusion systems described in \cite{AschCher10}.

\subsection{\protect\label{subsec:Clelland-Parker}The Clelland-Parker fusion
systems}

Clelland and Parker describe in \cite{ClellandParker10} two infinite
families of fusion systems which can be constructed as follows.

Let $p$ be an odd prime, let $q$ be a $p$-power, let $2\le n\le p-1$
be a natural number and set $k:=\Fp[q]$. The requirement $n\ge2$
does not exclude any of the Clelland-Parker exotic fusion systems
(see \cite[Theorems 5.1 and 5.2 and Lemma 5.3]{ClellandParker10}).

Define $D:=k^{*}\times\text{GL}_{2}\left(k\right)$ and $A:=A\left(n,k\right)$
as the group of homogenous polynomials in two variables, degree $n$
and coefficients in $k$. The group $D$ acts on $A$ via $f\left(x,y\right)\cdot\left(\lambda,\begin{pmatrix}a & b\\
c & d
\end{pmatrix}\right)=\lambda f\left(ax+by,cx+dy\right)$. This allows us to define the semidirect product $P:=P\left(n,k\right)=D\ltimes A$
and take a Sylow $p$-subgroup $S:=UA$ where $U\cong\left\{ \begin{pmatrix}1 & 0\\
a & 1
\end{pmatrix}\,:\,a\in k\right\} $ is seen as a Sylow $p$-subgroup of $D$. Finally, for every $i=0,\dots,n$,
set $C_{i}:=\left\{ \sum_{j=0}^{n}\lambda_{k}x^{n-j}y^{j}\right\} $
and define the subgroups $R=UC_{0}$ and $Q=UC_{1}$. Equivalently
$R$ and $Q$ are defined as
\begin{align*}
R & =\left\{ \left(1,\begin{pmatrix}1 & 0\\
a & 1
\end{pmatrix},\lambda x^{n}+\mu x^{n-1}y\right)\,:\,\lambda,\mu,a\in k\right\} ,\\
Q & =\left\{ \left(1,\begin{pmatrix}1 & 0\\
a & 1
\end{pmatrix},\lambda x^{n}+\mu x^{n-1}y+\nu x^{n-2}y\right)\,:\,\lambda,\mu,\nu,a\in k\right\} .
\end{align*}
As stated in \cite[page 296]{ClellandParker10} the normalizers of
$R$ and $Q$ in $P$ satisfy
\begin{align}
N_{P}\left(R\right) & =\left\{ \left(\theta,\begin{pmatrix}a & 0\\
c & b
\end{pmatrix},\lambda x^{n}+\mu x^{n-1}y\right)\,:\,\lambda,\mu,c\in k\,\wedge\,\theta,a,b\in k^{*}\right\} ,\label{eq:NPX}\\
N_{P}\left(Q\right) & =\left\{ \left(\theta,\begin{pmatrix}a & 0\\
c & b
\end{pmatrix},\lambda x^{n}+\mu x^{n-1}y+\nu x^{n-2}y\right)\,:\,\lambda,\mu,\nu,c\in k\,\wedge\,\theta,a,b\in k^{*}\right\} .\nonumber 
\end{align}

It turns out (see \cite[Lemma 4.6]{ClellandParker10}) that there
exist subgroups $P_{R}\le G_{R}:=\text{GL}_{3}\left(k\right)$ and
$P_{Q}\le G_{Q}:=\text{GSp}_{4}\left(k\right)$ and, for $X\in\left\{ R,Q\right\} $,
monomorphisms $\psi_{X}:N_{P}\left(X\right)\to P_{X}$ that map $N_{P}\left(X\right)$
to a Borel subgroup of $G_{X}$ and map $X$ to a normal subgroup
of $P_{X}$. Using these monomorphisms, Clelland and Parker construct
the amalgams 
\[
F_{X}:=F\left(1,n,k,X\right)=P*_{N_{P}\left(X\right)}P_{X}.
\]
Define $S_{X}:=N_{S}\left(X\right)\in\Syl_{p}\left(P_{X}\right)$,
$L:=N_{O^{p'}\left(P_{X}\right)}\left(S_{X}\right)O^{p'}\left(P\right)$
and $L_{X}:=\left(L\cap P_{X}\right)O^{p'}\left(P_{X}\right)$ viewed
as subgroups of $F_{X}$. Then we can further define the subamalgam
of $F_{X}$
\[
F_{X}^{*}:=F\left(q-1,n,k,X\right):=L*_{L\cap L_{X}}L_{X}.
\]
By viewing $F_{X}^{*}$ and $N_{P}\left(X\right)$ as subgroups of
$F_{X}$, it turns out (see \cite[page 296 and Lemma 4.7]{ClellandParker10})
that $F_{X}=N_{P}\left(X\right)F_{X}^{*}$ and that $F_{X}/F_{X}^{*}\cong N_{P}\left(X\right)/\left(L\cap L_{X}\right)\cong k^{*}$.
It follows that, for every $r|\left(q-1\right)$, we can take $x\in N_{P}\left(X\right)$
such that its projection $\overline{x}$ on $N_{P}\left(X\right)/\left(L\cap L_{X}\right)$
has order $\text{ord}\left(\overline{x}\right)=\frac{q-1}{r}$. We
can therefore define $P_{r}=L\left\langle x\right\rangle $ and $P_{X,r}:=L_{X}\left\langle x\right\rangle $
and, by viewing them as subgroups of $F_{X}$, we obtain the subamalgam
of $F_{X}$
\begin{equation}
F\left(r,n,k,X\right):=P_{r}*_{P_{r}\cap P_{X,r}}P_{X,r},\label{eq:realization-clelland-parker}
\end{equation}
which is the unique subgroup of $F_{X}$ of index $r$ containing
$F_{X}^{*}$. Finally the Clelland-Parker fusion systems are defined
as
\[
\F\left(r,n,k,X\right):=\FAB{S}{F\left(r,n,k,X\right)}.
\]

The following are useful properties of these fusion systems.
\begin{lem}
\label{lem:Clelland-Parker-X-and-A-are-centric}Let $X\in\left\{ R,Q\right\} $
and let $\F=\F\left(r,n,k,X\right)$. Then following hold:
\begin{enumerate}
\item \label{enu:Clelland-Parker-A-is-F-centric.}$A$ is $\F$-centric.
\item \label{enu:Clelland-Parker-F-conjugate-iff-F1-conjugate}For every
$Y\cong_{\F}X$ there exists $a\in P$ such that $Y=X^{a}$.
\item \label{enu:Clelland-Parker-X-is-F-centric}$X$ is $\F$-centric.
\end{enumerate}
\end{lem}

\begin{proof}
We know from \cite[Lemma 4.2(iv)]{ClellandParker10} that $A$ is
the unique abelian subgroup of $S$ of maximal order. In particular
it is both weakly $\F$-closed and $S$-centric. We conclude that
$A$ is $\F$-centric thus proving Item (\ref{enu:Clelland-Parker-A-is-F-centric.}).

Let $Y$ be an $\F$-conjugate of $X$. Since $\F$ is realized by
the amalgam in Equation (\ref{eq:realization-clelland-parker}), there
exist $n\ge0$ and elements $a_{0},\dots,a_{n}\in P_{r}$ and $b_{1},\dots,b_{n}\in P_{X,r}$
such that $Y=X^{a_{0}b_{1}a_{1}b_{2}\cdots a_{n}}$. Define $B_{0}:=X$
and, for every $i\ge1$, define $A_{i}:=B_{i-1}^{a_{i-1}}$ and $B_{i}:=A_{i}^{b_{i}}$.
In particular $Y=A_{n+1}$.

We know from \cite[Lemma 1]{Robinson07}, that $A_{i},B_{i}\le P_{r}\cap P_{X,r}\le N_{P}\left(X\right)$
for every $i=0,\dots,n$. Since $X$ is normal in $P_{X}$, if $a_{0}\in N_{P}\left(X\right)$
then $X=B_{1}$. Therefore, by taking $n$ minimal we can assume without
loss of generality that $a_{0}\not\in N_{P}\left(X\right)$. Using
the given description of $N_{P}\left(X\right)$ (see Equation (\ref{eq:NPX})),
it is easy to see that for every $a\in P$ either $a\in N_{P}\left(X\right)$
or $X^{a}\not\le N_{P}\left(X\right)$. We conclude that $A_{1}\not\le N_{P}\left(X\right)$
and, therefore, $n=0$ and $Y=A_{1}$. In particular $Y$ is conjugate
to $X$ via $a_{0}\in P$ thus proving Item (\ref{enu:Clelland-Parker-F-conjugate-iff-F1-conjugate}).

Since $X$ is $p$-centric in $P$ (see \cite[Lemma 4.4(ii)]{ClellandParker10})
then every $P$-conjugate of $X$ contained in $S$ is $S$-centric.
In particular Item (\ref{enu:Clelland-Parker-X-is-F-centric}) follows
from Item (\ref{enu:Clelland-Parker-F-conjugate-iff-F1-conjugate}).
\end{proof}
Let us recall the following well known result.
\begin{prop}
\label{prop:centric-radical-contains-normal}Let $S$ be a finite
$p$-group, let $\F$ be a saturated fusion system over $S$ and let
$Q\le S$ be normal in $\F$. Then any $\F$-centric-radical subgroup
of $S$ contains $Q$.
\end{prop}

\begin{proof}
See \cite[Proposition 4.46]{Craven11}.
\end{proof}
As a consequence we obtain the following.
\begin{lem}
\label{lem:clelland-parker-vanishing-subsystems}Let $\F=\F\left(r,n,k,X\right)$,
let $\mathcal{C}=\Fc$ be the family of $\F$-centric subgroups of
$S$, let $M^{*}$ be the contravariant part of a Mackey functor over
$\F$ with coefficients in $\Fp$ and write $M:=M^{*}\downarrow_{\OFC}^{\OF}$.
Then $\limn[i]_{\OFC[\mathcal{H}]}\left(M\downarrow_{\OFC[\mathcal{H}]}^{\OFC}\right)=0$
for every $i\ge1$ and every $\mathcal{H}\in\left\{ \FAB{S}{P_{r}},\FAB{S_{X}}{P_{r}\cap P_{X,r}},\FAB{S_{X}}{P_{X,r}}\right\} $.
\end{lem}

\begin{proof}
Since $X\le S_{X}$ is normal in $P_{X}$ (and therefore in $P_{X,r}$
and $P_{r}\cap P_{X,r}$) then it is normal in the fusion systems
$\FAB{S_{X}}{P_{X,r}}$ and $\FAB{S_{X}}{P_{r}\cap P_{X,r}}$. It
follows from Proposition \ref{prop:centric-radical-contains-normal},
that every $Y\in\left(\FAB{S_{X}}{P_{X,r}}\right)^{\text{cr}}\cup\left(\FAB{S_{X}}{P_{r}\cap P_{X,r}}\right)^{\text{cr}}$
contains $X$ and therefore, because of Lemma \ref{lem:Clelland-Parker-X-and-A-are-centric}(\ref{enu:Clelland-Parker-X-is-F-centric}),
is $\F$-centric.

Since $A$ is normal in $P$, it also follows from Lemma \ref{lem:Clelland-Parker-X-and-A-are-centric}(\ref{enu:Clelland-Parker-A-is-F-centric.})
that $\left(\FAB{S}{P_{r}}\right)^{\text{cr}}\subseteq\Fc$.

It follows from \cite[Proposition 10.5]{Yal22} (see also \cite[Corollary 3.6]{BLO03})
that $\limn[i]_{\OFC[\mathcal{H}]}M\downarrow_{\OFC[\mathcal{H}]}^{\OFC}=\limn[i]_{\OFc[\mathcal{H}]}M\downarrow_{\OFc[\mathcal{H}]}^{\OFC}$
for every $i\ge0$. Since $\mathcal{H}$ is a realizable fusion system,
the result follows from \cite[Theorem B]{DiazPark15}.
\end{proof}
Fix $r$, $n$, $k$ and $X$ and define $\F=\F\left(r,n,k,X\right)$.
Adopting Notation \ref{nota:simple-tree-of-fusion-systems}, let $S'=S_{X}$,
$G_{\boldsymbol{1}}=P_{r}$, $G_{\boldsymbol{2}}=P_{X,r}$ and $G_{\boldsymbol{e}}=P_{r}\cap P_{X,r}$.
In particular we have that
\begin{align*}
\F_{\boldsymbol{1}} & =\FAB{S}{P_{r}}, & \F_{\boldsymbol{2}} & =\FAB{S_{X}}{P_{X,r}}, & \F_{\boldsymbol{e}} & =\FAB{S_{X}}{P_{r}\cap P_{X,r}},
\end{align*}
and $G=\G_{\T}=F\left(r,n,k,X\right)$. Let $\tilde{\T}$ be the orbit
graph associated to the tree of groups $\left(\T,\G\right)$. Then
the following result holds
\begin{lem}
\label{lem:Clelland-Parker-finite-fixed-point-tree}The tree $\tilde{\T}^{Y}$
is finite for every $Y\in\Fc$.
\end{lem}

\begin{proof}
This is precisely what is proven in \cite[Theorem 4.9]{ClellandParker10}.
\end{proof}
We can now prove the first half of Theorem \hyperref[thm:C]{C}.
\begin{prop}
\label{prop:shar-Clelland-Parker}Let $\F$ be a Clelland-Parker fusion
system, let $M^{*}$ be the contravariant part of a Mackey functor
over $\F$, with coefficients in $\Fp$ and write $M:=M^{*}\downarrow_{\OFc}^{\OF}$.
Then $\limn[i]_{\OFc}\left(M\right)=0$ for every $i\ge2$ while 
\[
\limn[1]_{\OFc}\left(M\right)=M^{\F_{\boldsymbol{e}}}/\left(\overline{\iota_{S'}^{S}}\left(M^{\F_{\boldsymbol{1}}}\right)+M^{\F_{\boldsymbol{2}}}\right).
\]
Moreover $\underset{\OFc}{\limn[i]}\left(H^{j}\left(-;\Fp\right)\right)=0$
for every $i\ge1$ and every $j\ge0$.
\end{prop}

\begin{proof}
We know from Lemma \ref{lem:Clelland-Parker-finite-fixed-point-tree}
that $\tilde{\T}^{Y}$ is finite for every $Y\in\Fc$. It follows
from \cite[Lemma 3.3]{ClellandParker10} that $\tilde{\T}^{Y}/C_{G}\left(Y\right)$
is a tree and, therefore, $\CGp\left(Y\right)=0$. The result follows
from Theorem \hyperref[thm:B]{B}.
\end{proof}

\subsection{\protect\label{subsec:Parker-Stroth}The Parker-Stroth fusion systems.}

In \cite{ParkerStroth15} Parker and Stroth describe a family of exotic
fusion systems. Like the Clelland-Parker fusion systems described
in Subsection \ref{subsec:Clelland-Parker} this family is related
to an action on homogeneous polynomials on two variables. However,
the notation used in \cite{ParkerStroth15} differs from that used
in \cite{ClellandParker10} and Subsection \ref{subsec:Clelland-Parker}.
In order to be consistent with our notation we slightly modify the
notation of \cite{ParkerStroth15} in the following description of
the Parker-Stroth fusion systems. The reader who wishes to compare
the present document with \cite{ParkerStroth15} might find the following
notation conversion table useful.
\begin{center}
\begin{tabular}{|l|c|c|}
\hline 
Description & Our notation & Notation in \cite{ParkerStroth15}\tabularnewline
\hline 
\hline 
\multirow{1}{*}{field of $k$ elements} & $\Fp[k]$ or $k$ & $\mathbb{F}$\tabularnewline
\hline 
group of invertible elements & $\Fp[k]^{*}$ or $k^{*}$ & $\mathbb{F}^{\times}$\tabularnewline
\hline 
\multirow{2}{*}{\begin{cellvarwidth}[t]
degree $n$ homogeneous polynomials\\
over $k$ in $2$ variables
\end{cellvarwidth}} & \multirow{2}{*}{$A\left(n,k\right)$} & \multirow{2}{*}{$V_{n}$}\tabularnewline
 &  & \tabularnewline
\hline 
degree of homogeneous polynomials & $n$ & $m$\tabularnewline
\hline 
group $k^{*}\times\text{GL}_{2}\left(k\right)$ & $D$ & $L$\tabularnewline
\hline 
Sylow $p$-subroup of $D$ & $U$ & $S_{0}$\tabularnewline
\hline 
\end{tabular}
\par\end{center}

The usage of the symbols $P$, $X$, $G$, $S$ and $Q$ in \cite{ParkerStroth15}
and Subsection \ref{subsec:Clelland-Parker} also differ. In this
subsection we will give $G$, $S$ and $Q$ the meaning used in \cite{ParkerStroth15}
and we will limit the usage of the other symbols in order to avoid
confusions.

\smallskip{}

Let $p\ge5$ be a prime, let $n=p-4$, define $A:=A\left(n,p\right)$
and define $\beta_{n}:A\times A\to p$ as the bilinear form satisfying
\[
\beta_{n}\left(x^{a}y^{b},x^{c}y^{d}\right)=\begin{cases}
\frac{\left(-1\right)^{a}}{\binom{n}{a}} & \text{if }a=d\\
0 & \text{else}
\end{cases}.
\]
The set $Q:=A\times\Fp$ can be given a non-abelian group structure
by defining the product as
\[
\left(v,y\right)\left(w,z\right)=\left(v+w,y+z+\beta_{n}\left(v,w\right)\right).
\]
Define the group $D:=\Fp^{*}\times\text{GL}_{2}\left(\Fp\right)$
as in Subsection \ref{subsec:Clelland-Parker}. There exists a right
action of $D$ on $Q$ obtained by defining for every $\left(v,y\right)\in Q$
and every $\left(t,M\right)\in D$
\[
\left(v,y\right)\cdot\left(t,M\right):=\left(v\cdot\left(t,M\right),t^{2}\det\left(M\right)^{n}z\right),
\]
here the action of $D$ on $A$ is as in Subsection \ref{subsec:Clelland-Parker}.
In \cite[Lemma 2.3 and Theorem 2.9]{ParkerStroth15} Parker and Stroth
prove that 
\begin{equation}
C_{D}\left(Q\right)=\left\{ \left(\mu^{-n},\begin{pmatrix}\mu & 0\\
0 & \mu
\end{pmatrix}\right)\,:\,\mu\in\Fp^{*}\right\} ,\label{eq:description-CDQ}
\end{equation}
 and that the group $P_{1}:=\left(D\ltimes Q\right)/C_{D}\left(Q\right)$
contains a subgroup isomorphic to the subgroup
\[
C:=\left\{ \begin{pmatrix}1 & 0 & 0\\
\alpha & \theta & 0\\
\delta & \varepsilon & \theta^{-1}
\end{pmatrix}\in\text{SL}_{3}\left(\Fp\right)\right\} ,
\]
of
\[
K:=\left\{ \begin{pmatrix}1 & 0 & 0\\
\alpha & \beta & \gamma\\
\delta & \varepsilon & \phi
\end{pmatrix}\in\text{SL}_{3}\left(\Fp\right)\right\} .
\]
Using this isomorphism, it is possible to define the amalgam $G=P_{1}*_{C}K$
and view $K$, $C$ and $P_{1}$ as subgroups of $G$. Taking $U\in\Syl_{p}\left(D\right)$
as in Subsection \ref{subsec:Clelland-Parker}, we can further define
$S:=\left(UQC_{D}\left(Q\right)\right)/C_{D}\left(Q\right)\le P_{1}$
and $S':=S\cap C$ (denoted as $D$ in \cite{ParkerStroth15}). It
turns out (see \cite[page 320]{ParkerStroth15}) that $S$ is a Sylow
$p$-subgroup of $P_{1}$ and that $S'\cong p_{+}^{1+2}$ is a Sylow
$p$-subgroup of both $C$ and $K$. Moreover, the fusion system $\F:=\FSG$
is saturated (see \cite[Lemma 3.2(iii)]{ParkerStroth15}) and exotic
(see \cite[Lemma 3.4]{ParkerStroth15}). We call any fusion system
$\F$ constructed in this manner a Perker-Stroth fusion system.

With this notation in place we can now prove the following.
\begin{lem}
\label{lem:parker-stroth-vanishing-subsystems}Let $\mathcal{C}$
be the family of $\F$-centric subgroups of $S$, let $M^{*}$ be
the contravariant part of a Mackey functor over $\F$ with coefficients
in $\Fp$ and write $M:=M^{*}\downarrow_{\OFC}^{\OF}$. Then $\limn[i]_{\OFC[\mathcal{H}]}\left(M\downarrow_{\OFC[\mathcal{H}]}^{\OFC}\right)=0$
for every $i\ge1$ and every $\mathcal{H}\in\left\{ \FSG[P_{1}],\FAB{S'}{C},\FAB{S'}{K}\right\} $.
\end{lem}

\begin{proof}
Following the notation of \cite[page 320]{ParkerStroth15}, we define
$W:=O_{p}\left(K\right)$. Since $W$ is a normal subgroup of $K$
and $W\le S'$, it follows that $W$ is normal in $\FAB{S'}{K}$.
We conclude from Proposition \ref{prop:centric-radical-contains-normal},
that every subgroup in $\left(\FAB{S'}{K}\right)^{\text{cr}}$ contains
$W$. Since $W$ $\F$-centric (see \cite[Lemma 3.3]{ParkerStroth15}),
it follows that $\left(\FAB{S'}{K}\right)^{\text{cr}}\subseteq\Fc$.
Since $W\le C\le K$, the same arguments prove that every $\left(\FAB{S'}{C}\right)^{\text{cr}}\subseteq\Fc$.

We know from \cite[page 320]{ParkerStroth15} that $\left(QC_{D}\left(Q\right)\right)/C_{D}\left(Q\right)=O_{p}\left(P_{1}\right)$.
Since $C_{D}\left(Q\right)$ is a $p$-prime group (see Equation (\ref{eq:description-CDQ})),
we have that $S\cong UQ$ via the natural isomorphism which also sends
$\left(QC_{D}\left(Q\right)\right)/C_{D}\left(Q\right)$ to $Q$.
A straightforward computation shows that $Q$ is centric in $UQ$.
We conclude that $O_{p}\left(P_{1}\right)$ is $\FSG[P_{1}]$-centric.
Since $\left|O_{p}\left(P_{1}\right)\right|\ge\left|S'\right|$ and
$\F=\left\langle \FSG[P_{1}],\FAB{S'}{K}\right\rangle $ (see \cite[Theorem 1]{Robinson07}),
then every subgroup of $S$ conjugate to $O_{p}\left(P_{1}\right)$
in $\F$ is also conjugate in $\FSG[P_{1}]$. In particular $O_{p}\left(P_{1}\right)$
is $\F$-centric. Since $O_{p}\left(P_{1}\right)$ is also normal
in $\FSG[P_{1}]$ (because it is normal in $P_{1}$), then we conclude
from Proposition \ref{prop:centric-radical-contains-normal} that
every $\left(\FSG[P_{1}]\right)^{\text{cr}}\subseteq\Fc$.

It follows from \cite[Proposition 10.5]{Yal22} (see also \cite[Corollary 3.6]{BLO03})
that $\limn[i]_{\OFC[\mathcal{H}]}M\downarrow_{\OFC[\mathcal{H}]}^{\OFC}=\limn[i]_{\OFc[\mathcal{H}]}M\downarrow_{\OFc[\mathcal{H}]}^{\OFC}$
for every $i\ge0$. Since every $\mathcal{H}$ is a realizable fusion
system, then the result follows from \cite[Theorem B]{DiazPark15}.
\end{proof}
Proceeding as in Subsection \ref{subsec:Clelland-Parker} we can now
adopt Notation \ref{nota:simple-tree-of-fusion-systems} with $G_{\boldsymbol{1}}=P_{1}$,
$G_{\boldsymbol{2}}=K$ and $G_{\boldsymbol{e}}=C$. In particular
we have that
\begin{align*}
\F_{\boldsymbol{1}} & :=\FSG[P_{1}], & \F_{\boldsymbol{2}} & :=\FAB{S'}{C}, & \F_{\boldsymbol{e}} & :=\FAB{S'}{K},
\end{align*}
and $G=\G_{\T}=P_{1}*_{C}K$. Let $\tilde{\T}$ be the orbit graph
associated to the tree of groups $\left(\T,\G\right)$. Then the following
holds
\begin{lem}
\label{lem:quotient-is-tree-Parker-Stroth}The tree$\tilde{\T}^{Y}$
is finite for every $Y\in\Fc$.
\end{lem}

Proof. Since every $\F$-centric subgroup of $S$ has order at least
$2$, the result follows from \cite[Lemma 3.2]{ParkerStroth15}.

We can now prove the second half of Theorem \hyperref[thm:C]{C}.
\begin{prop}
\label{prop:shar-Parker-Stroth}Let $\F$ be a Parker-Stroth fusion
system, $M^{*}$ be the contravariant part of a Mackey functor over
$\F$ with coefficients in $\Fp$ and write $M:=M^{*}\downarrow_{\OFc}^{\OF}$.
Then $\limn[i]_{\OFc}\left(M\right)=0$ for every $i\ge2$ while 
\[
\limn[1]_{\OFc}\left(M\right)=M^{\F_{\boldsymbol{e}}}/\left(M^{\F_{\boldsymbol{1}}}\overline{\iota_{S'}^{S}}+M^{\F_{\boldsymbol{2}}}\right).
\]
Moreover $\underset{\OFc}{\limn[i]}\left(H^{j}\left(-;\Fp\right)\right)=0$
for every $i\ge1$ and every $j\ge0$.
\end{prop}

\begin{proof}
We know from Lemma \ref{lem:quotient-is-tree-Parker-Stroth} that
$\tilde{\T}^{Y}$ is finite for every $Y\in\Fc$. It follows from
\cite[Lemma 3.3]{ClellandParker10} that $\tilde{\T}^{Y}/C_{G}\left(Y\right)$
is a tree and, therefore, $\CGp\left(Y\right)=0$. The result follows
from Theorem \hyperref[thm:B]{B}.
\end{proof}
\begin{rem}
Proposition \ref{prop:shar-Parker-Stroth} is a special case of \cite[Theorem A]{GraMar23}
and the classification of such fusion systems given in \cite[Theorem 1.1]{ParkerSemeraro18}.
See also \cite[Theorems 1.1 and 5.2(d)]{HLL23} for an alternative
proof.
\end{rem}

\subsection{\protect\label{subsec:Benson-Solomon}The Benson-Solomon fusion systems}

The Benson-Solomon fusion systems are, up to date, the only known
family of exotic fusion systems over the prime $2$. Following the
group theoretic approach taken by Aschbacher and Chermak in \cite{AschCher10},
they can be defined as follows.

Let $\overline{\Ff}$ be the algebraic closure of the finite field
of $5$ elements and let $\Ffb\subseteq\overline{\Ff}$ be the union
of all the subfields of $\overline{\Ff}$ of order $5^{2^{n}}$ for
some $n\ge0$. Define $B^{0}:=\left(\text{SL}_{2}\left(\Ffb\right)\right)^{3}/\left\langle -\left(\Id,\Id,\Id\right)\right\rangle $,
let $S_{3}$ act on $B^{0}$ on the obvious way and define the semidirect
products
\begin{align*}
B & :=\left\langle \left(1,2\right)\right\rangle \ltimes B^{0}, & K & :=\Sym\left(3\right)\ltimes B^{0},
\end{align*}
so that we view $B$ as a subgroup of $K$. It is possible (see \cite[Theorem A]{AschCher10})
to take a group $H$ of $\Ffb$-rational points in $\Spin_{7}\left(\overline{\Ff}\right)$
and a monomorphism $\lambda^{*}:B\to H$ such that the resulting amalgam
$G:=H*_{B}K$ admits an automorphism $\psi_{0}$ satisfying the following:
\begin{enumerate}
\item Every power of $\psi_{0}$ of the form $\sigma=\psi_{0}^{2^{n}}$,
leaves invariant each of the subgroups $H$, $K$ and $B$ of $G$.
In particular, $G_{\sigma}=H_{\sigma}*_{B_{\sigma}}K_{\sigma}$ where
the subindex $\sigma$ denotes the subgroup of elements fixed under
the action of $\sigma$.
\item The group $H_{\sigma}$ is isomorphic to $\Spin_{7}\left(\Fp[5^{2^{n}}]\right)$,
the subgroups $B_{\sigma}$ and $K_{\sigma}$ are finite, $K_{\sigma}=B_{\sigma}\left\langle \left(1,2,3\right)\right\rangle $
and $B_{\sigma}=B_{\sigma}^{0}\left\langle \left(1,2\right)\right\rangle $.
\item There exists a $2$-subgroup $S\le B$ such that $S_{\sigma}\le B$
is a Sylow $2$-subgroup of $B_{\sigma}$, $K_{\sigma}$ and $H_{\sigma}$.
\item The fusion system $\F_{\sigma}=\FAB{S_{\sigma}}{G_{\sigma}}$ is the
Benson-Solomon fusion system $\F_{\Sol}\left(5^{2^{n}}\right)$. In
particular, all Benson-Solomon fusion systems are of this form (see
\cite[Theorem B]{COS08}).
\item There exists a normal subgroup $X\trianglelefteq G$, such that, for
every subgroup $P$ of $S_{\sigma}$, the set 
\[
\theta_{\sigma}\left(P\right):=C_{X_{\sigma}}\left(P\right)\mathcal{O}\left(C_{G_{\sigma}}\left(P\right)\right),
\]
is a subgroup of $C_{G_{\sigma}}\left(P\right)$. Moreover $\theta_{\sigma}\left(P\right)$
is a complement to $Z\left(P\right)$ in $C_{G_{\sigma}}\left(P\right)$
(see \cite[Theorem 8.8]{AschCher10}). In other words $\theta_{\sigma}$
is a signalizer functor for $G_{\sigma}$ (see \cite[Definition 2.5]{AschCher10}).
\end{enumerate}
With this setup we obtain the following.
\begin{lem}
\label{lem:Benson-Solomon-vanishing-subsystems}Let $\mathcal{C}$
be the family of $\F_{\sigma}$-centric subgroups of $S_{\sigma}$,
let $M^{*}$ be the contravariant part of a Mackey functor over $\F_{\sigma}$
with coefficients in $\Fd$ and write $M:=M^{*}\downarrow_{\OFC}^{\OF}$.
Then $\limn[i]_{\OFC[\mathcal{H}]}M\downarrow_{\OFC[\mathcal{H}]}^{\OFC}=0$
for every $i\ge1$ and every $\mathcal{H}\in\left\{ \FAB{S_{\sigma}}{H_{\sigma}},\FAB{S_{\sigma}}{B_{\sigma}},\FAB{S_{\sigma}}{K_{\sigma}}\right\} $.
\end{lem}

\begin{proof}
In order to simplify notation, throughout this proof we write $\F_{X}=\FAB{S_{\sigma}}{X_{\sigma}}$
for $X\in\left\{ H,B,K\right\} $ and, for every $A,B,C\in\text{SL}_{2}\left(\Ffb\right)$,
we denote by $\left\llbracket A,B,C\right\rrbracket $ any element
in the quotient group $B^{0}$.

Let $Z=Z\left(S_{\sigma}\right)=\left\llbracket \Id,\Id,-\Id\right\rrbracket $.
We know from \cite[Proposition 9.2]{AschCher10} that $\F_{H}=C_{\F_{\sigma}}\left(Z\right)$.
Since $B_{\sigma}=K_{\sigma}\cap H_{\sigma}$ and the element $\left(1,2,3\right)$
of $K_{\sigma}$ does not centralize $Z$, then we conclude that $\F_{B}=C_{\F_{K}}\left(Z\right)$.
It follows from \cite[Proposition 2.5]{BLO03} that $\F_{\sigma}^{\text{c}}=\F_{H}^{\text{c}}$
and that $\F_{K}^{\text{c}}=\F_{B}^{\text{c}}$. In particular, since
$\F_{H}\subseteq\F_{\sigma}$ and $\F_{B}\subseteq\F_{K}$, then $\F_{H}^{\text{cr}}\subseteq\F_{\sigma}^{\text{cr}}\subseteq\F_{\sigma}^{\text{c}}$
and $\F_{B}^{\text{cr}}\subseteq\F_{K}^{\text{cr}}$. The elements
in $\F_{K}^{\text{cr}}$ are listed in \cite[Lemma 10.2]{AschCher10}.
On the other hand \cite[Lemma 10.9]{AschCher10} ensures us that these
subgroups are also $\F_{\sigma}$-centric-radical and, in particular,
$\F_{\sigma}$-centric.

It follows from \cite[Proposition 10.5]{Yal22} (see also \cite[Corollary 3.6]{BLO03}),
that $\limn[i]_{\OFC[\mathcal{H}]}M\downarrow_{\OFC[\mathcal{H}]}^{\OFC}=\limn[i]_{\OFc[\mathcal{H}]}M\downarrow_{\OFc[\mathcal{H}]}^{\OFC}$
for every $i\ge0$. Since every $\mathcal{H}$ is a realizable fusion
system, then the result follows from \cite[Theorem B]{DiazPark15}.
\end{proof}
\begin{rem}
The reader might not find it immediate that the $\F_{K}$-centric-radical
subgroups of $S_{\sigma}$ listed in \cite[Lemma 10.2]{AschCher10}
are also listed in \cite[Lemma 10.9]{AschCher10}. An alternative
proof of the fact that all $\F_{K}$-centric-radical subgroups of
$S_{\sigma}$ are $\F_{\sigma}$-centric-radical can be obtained via
the exhaustive list of all such subgroups appearing in \cite[Tables 1, 2 and 4]{LyndSemeraro23}.

Lemma \ref{lem:Benson-Solomon-vanishing-subsystems} allows us to
apply Theorem \hyperref[thm:A]{A} for the Benson-Solomon fusion systems.
By definition of signalizer functor we also know that $\theta_{\sigma}\left(P\right)\cong C_{G_{\sigma}}\left(P\right)/Z\left(P\right)$.
Theorem \hyperref[thm:D]{D} now follows from Proposition \ref{prop:first-homology}.
\end{rem}

\printbibliography

\end{document}